\documentclass{article}

    \PassOptionsToPackage{numbers, compress}{natbib}

\usepackage[final]{neurips_2022}
\usepackage{algorithm}
\usepackage{algpseudocode}
\usepackage{caption}

\usepackage[utf8]{inputenc} 
\usepackage[T1]{fontenc}    
\usepackage{hyperref}       
\usepackage{url}            
\usepackage{booktabs}       
\usepackage{amsfonts}       
\usepackage{nicefrac}       
\usepackage{microtype}      

\usepackage{hyperref}       %
\usepackage{graphicx}
\usepackage{subfigure}
\usepackage{amsmath}
\usepackage{amsthm}
\usepackage{amssymb}
\usepackage{bbm}
\usepackage{wrapfig}
\usepackage{nicefrac}       %

\usepackage[dvipsnames]{xcolor}

\theoremstyle{plain}
\newtheorem{prop}{\protect\propositionname}
\theoremstyle{plain}
\newtheorem{thm}{\protect\theoremname}
\theoremstyle{plain}
\newtheorem{assumption}{\protect\assumptionname}
\theoremstyle{plain}
\newtheorem{lem}[thm]{\protect\lemmaname}
\theoremstyle{plain}

\providecommand{\assumptionname}{Assumption}
\providecommand{\lemmaname}{Lemma}
\providecommand{\propositionname}{Proposition}
\providecommand{\corollaryname}{Corollary}
\providecommand{\theoremname}{Theorem}

\newtheorem{defi}{Definition}
\newcommand{\red}[1]{\textcolor{black}{#1}}
\include{notation_song}

\title{Estimating the Arc Length of the Optimal ROC Curve and Lower Bounding the Maximal AUC}

\begin{document}

\author{%
  Song Liu\thanks{Webpage: \url{https://allmodelsarewrong.net}} \\
  School of Mathematics\\
  University of Bristol\\
  Bristol, BS8 1UG, UK\\
  \texttt{song.liu@bristol.ac.uk} \\
}

\maketitle

\begin{abstract}
\red{
In this paper, we show the arc length of the optimal ROC curve is an $f$-divergence. By leveraging this result, we express the arc length using a variational objective and estimate it accurately using positive and negative samples. 
We show this estimator has a non-parametric convergence rate $O_p(n^{-\beta/4})$ ($\beta \in (0,1]$ depends on the smoothness). 
Using the same technique, we show the surface area between the optimal ROC curve and the diagonal can be expressed via a similar variational objective. 
These new insights lead to 
a novel classification procedure that maximizes an approximate lower bound of the maximal AUC. 
Experiments on CIFAR-10 datasets show the proposed two-step procedure achieves good AUC performance in imbalanced binary classification tasks.
}
\end{abstract}
\section{Introduction}

The study of Receiver operating characteristic (ROC) curves has a long history in medicine \citep{lusted1971signal}, psychology \citep{green1966signal} and radiology \citep{hanley1982meaning,goodenough1974radiographic}. 
In machine learning, ROC curves have been primarily used to analyze the performance of different classification algorithms \citep{Fawcett2006,flach2016roc}. 
Indeed, the Area Under the Curve (AUC) encodes a classifier's ranking accuracy, making it a preferable performance metric for imbalanced class classification \cite{Fawcett2006,cortes2003}. 
%
In recent years,
ROC curves have also been used in comparing two distributions and achieved promising results. Examples include analyzing the mode collapsing issue of Generative Adversarial nets (GAN) \citep{Lin2018}, and diagnosing the performance of an amortized Markov Chain Monte Carlo \citep{hermans2020likelihood}. 

In applications that require computing statistical discrepancy between distributions (e.g. GAN \cite{Goodfellow2014} or Variational Inference (VI) \cite{Blei2017}), $f$-divergences are widely used discrepancy measures. The family of $f$-divergences includes Kullback-Leibler divergence \cite{Kullback1951} and Total Variation distance. It has been shown that $f$-divergences, generally, can be expressed via variational objectives and efficiently approximated from empirical samples \cite{Nguyen2010,Nowozin2016}. 

Since the ROC curves are used as performance metrics in many two sample applications, are they in any way related to $f$-divergences?
For example, can AUC be an $f$-divergence between positive and negative data distributions given some classification score function? An earlier investigation proves that the answer is no when the score function is the likelihood ratio \citep{reid11a}. Nonetheless, this result inspired us to look for $f$-divergence from other geometries of the ROC curve. 

In this paper, we show that, when using the likelihood ratio score, a novel $f$-divergence arises from the \emph{arc length} of the corresponding ROC curve. By leveraging this result, we can express the arc length using a variational objective and approximate it using only samples from two distributions.
We show this arc length estimator is also a consistent estimator to the arctangent of likelihood ratio and has a non-parametric convergence rate $O_p(n^{-\beta/4})$, where $\beta\in (0,1]$ depends on the smoothness of the true arctangent likelihood ratio. 
Moreover, by parameterizing the ROC curves of positive and negative \emph{mixtures} distributions, 
the surface area \red{between the optimal ROC curve and the diagonal} can be expressed via a similar variational objective. With the help of our arctangent ratio estimator, we can approximately maximize a lower bound to this surface area.  
We point out the similarity between this lower bound maximization and the classic AUC maximization  \cite{cortes2003}.
We show our approximated optimal score achieves comparable performance to a state of the art AUC maximizer in an imbalanced classification problem on CIFAR-10 dataset.
%


\section{Background}
\subsection{An Illustrative Example}
\begin{figure*}[t]
    \centering
\subfigure[different distributions, long ROC]{\includegraphics[width=.32\textwidth]{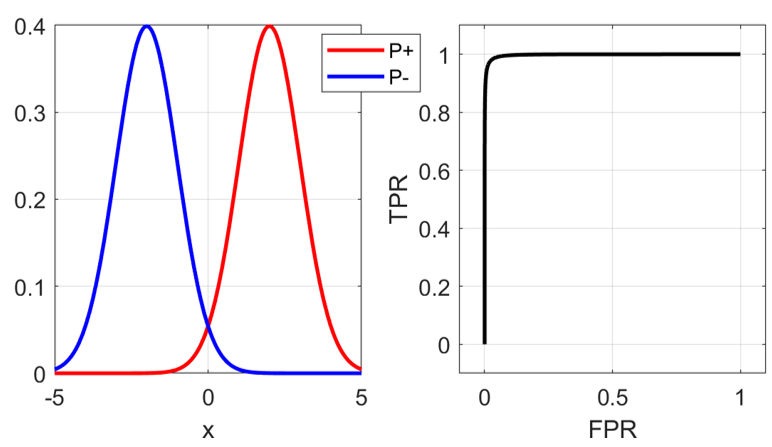}
\label{fig.goodranking.largeauc}
}
\subfigure[different distributions, long ROC]{\includegraphics[width=.32\textwidth]{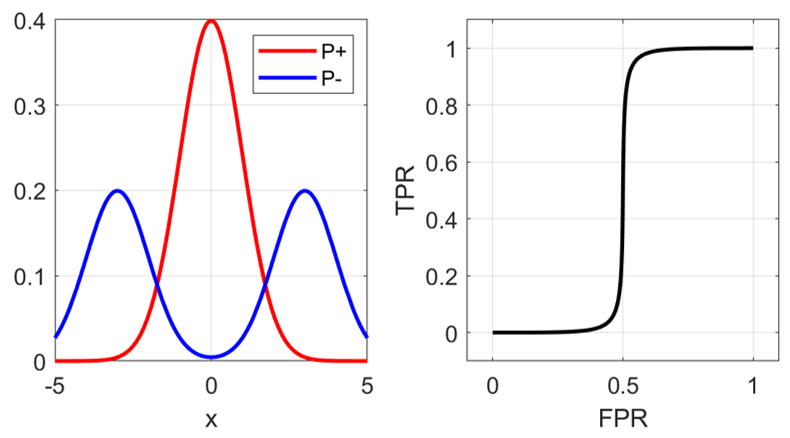}
\label{fig.poorranking1}
}
\subfigure[same distribution, short ROC]{\includegraphics[width=.32\textwidth]{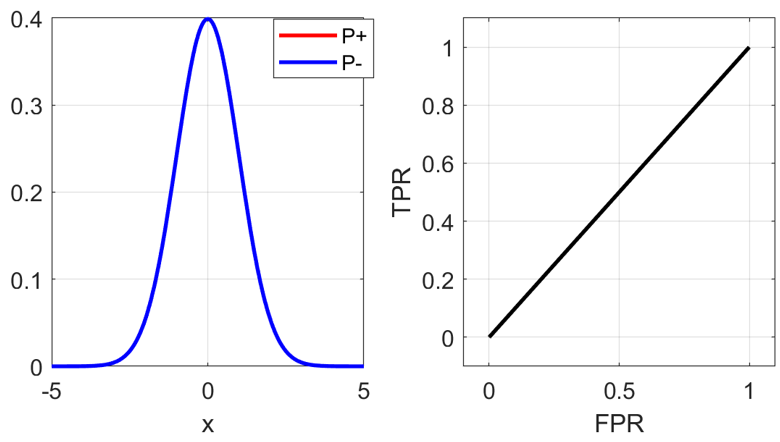}
\label{fig.poorranking2}
}
\caption{
    ROC curves generated for one dimensional datasets using the identity classification score function $t(x) = x$. Notice that
    the arc length of ROC curves seem to be a good indication on how different the positive (red) and negative (blue) data distributions are.  
    }
\label{fig.ranking.vs.overlapping}
\end{figure*}

ROC curves are frequently used to visualize binary classification performance, and 
we often rely on the Area Under the Curve (AUC) as a numerical metric for selecting a good classifier. In this section, we highlight an often overlooked ROC geometry: arc length. We illustrate its potential as a good discrepancy measure between positive and negative data distributions.  

Let us consider the one-dimensional distributions and the ROC curves in Figure \ref{fig.ranking.vs.overlapping}. In this example, ROC curves are generated using the identity score function $t(x) = x$.

In both (a) and (b), $p_+$ and $p_-$ are quite different, and thus the discrepancies between positive and negative distributions should be large in both cases. 
In (c), the densities $p_+$ and $p_-$ are the same and consequently their discrepancy should be smaller than that of both (a) and (b).
Further, notice that 
both (a) and (b) have long ROC curves ($\approx 2$) while (c) has a shorter one ($\approx \sqrt{2}$). 
This example suggests that 
the more similar $p_+$ and $p_-$ are, the shorter the ROC curve is.

However, we can see that $t(x) = x$ is a special choice:
If $t(x)=0$, the arc length will not reflect any discrepancy between data distributions at all. 
This observation inspires the following questions: 
Why is the arc length in this example good at telling the differences between two data distributions? 
Are there other score functions whose ROC arc lengths are also good discrepancy measures? What are the practical applications of studying the arc length of ROC? 
In the following sections, we study the arc length of a ROC curve under a probabilistic framework and provide answers to these questions.  


\subsection{ROC Curve in a Probabilistic Setting}
Suppose we have positive and negative datasets $X_+ := \{\boldx^+_i\}_{i=1}^{n_+}$ and $X_- :=\{\boldx^-_i\}_{i=1}^{n_-}$ drawn from two distributions $\mathbb{P}_+$ and $\mathbb{P}_-$ respectively. These distributions have respective probability density functions $p_+(\boldx)$ and $p_-(\boldx)$ that are both defined on the domain $\mathcal{X}\subseteq \mathbbR^d$. 
\red{A classification score function (score function for short)} takes a sample $\boldx$ as input and outputs a real-valued score. Suppose we have a score function $t(\boldx) \in \mathbb{R}$. 
\red{We classify $\boldx$ as positive if $t(\boldx) > \tau$, where $\tau$ is a threshold.} 

Let $F_+$ and $F_-$ denote the Cumulative Distribution Functions (CDFs) of $t(\boldx^+)$ and $t(\boldx^-)$ respectively. 
Then we can define the following quantities: 
\begin{itemize}
    \item 
    False Positive Rate (FPR) \red{at threshold $\tau$}, $\tilde{F}_-(\red{\tau}) := 1 - F_-(\red{\tau})$
    \item 
    True Positive Rate (TPR) \red{at threshold $\tau$}, $\tilde{F}_+(\red{\tau}) := 1 - F_+(\red{\tau})$
    \item 
    The ROC curve of a score function $t$: the graph of function $\tilde{F}_+[\tilde{F}_-^{-1}(s)]$, where $s \in [0,1]$. 
\end{itemize}
The above definition of ROC curve requires $F_-$ to be strictly increasing. In this paper, we assume $F_+$, $F_-$ to be both strictly increasing. Obviously, both $F_+$ and $F_-$ depend on the choice of score function $t$. 

\subsection{Arc Length of ROC Curve}
\label{sec.arclenth}

Due to the strict monotonicity of $F_+$ and $F_-$, $\tilde{F}_+$ and $\tilde{F}_-$ form a bijective parameterization of the ROC curve in the sense that 
each point on this ROC curve can be written as $(\tilde{F}_-(\tau_0), \tilde{F}_+(\tau_0))$ for a unique $\tau_0 \in \mathbbR$.
Using the line integral formula, the arc length of an ROC  curve for a fixed score function $t$ can be expressed using the derivatives of $\tilde{F}_-$ and $\tilde{F}_+$: 
\begin{align}
    \label{eq.length00}
    \arc{\mathrm{ROC}}(t) &:= \int_{-\infty}^{\infty}  \sqrt{\left[\partial_{\tau}  \tilde{F}_+(\tau)\right]^2 +   \red{\left[\partial_{\tau}\tilde{F}_-(\tau)\right]^2}} \mathrm{d} \tau
    = \int_{-\infty}^{\infty} \sqrt{ f_{+}(\tau)^2 +   f_{-}(\tau)^2} \text{d} \tau,
\end{align}
where $f_{+}(\tau)$ and $f_{-}(\tau)$ are the density functions of $t(\boldx^{+})$ and $t(\boldx^{-})$ respectively. 
Although \eqref{eq.length00} is an elementary result, it has been seldom discussed in the ROC literature. 
%
Authors in \cite{edwards2007utility, edwards2008optimality} have proposed a performance metric computed over the ``ROC hypersurface'' and \eqref{eq.length00} is used to justify such a metric in a binary classification setting. 

Using \eqref{eq.length00}, we can confirm a simple geometric fact whose proof can be found in Appendix  \ref{sec.poof.prop.roc}: 
\begin{prop}
    \label{prop.roc}
    $\arc{\mathrm{ROC}}(t) \in [\sqrt{2},2]$, for all $t$. 
\end{prop}

This result reflects the geometric observation that any monotone curve (such as ROC curve) starts and ends at two opposite corners of the ROC space $[0,1]^2$ has an arc length between $\sqrt{2}$ and $2$.
%



\section{$f$-divergences Arising from ROC Arc Length}
\label{sec.fdiv}
\subsection{$f$-divergence of Score Distributions}
\label{sec.fdiv.on.score}
Among many discrepancies measures, $f$-divergence has been widely used in many applications.
\begin{defi} Let $p$ and $q$ be densities of two continuous distributions. 
An $f$-divergence is defined as: 
    $\mathrm{D}_g(p|q):= \mathbb{E}_{q} \left[ g\left( \frac{p(\boldz)}{q(\boldz)} \right) \right]$, 
    where $g$ is convex and lower-semicontinuous satisfying $g(1)= 0$. 
\end{defi}
Now let us slightly rewrite \eqref{eq.length00}. Assuming $f_{-}$ is strictly positive (in which case $F_-$ is strictly increasing), we can write 
\begin{align}
    \label{eq.roc.length.0}
    \arc{\mathrm{ROC}}(t) - \sqrt{2} 
    = \mathbb{E}_{f_-} \sqrt{\left[\frac{f_{+}(\tau)}{f_{-}(\tau)} \right]^2 + 1}  - \sqrt{2}. 
    = \mathbb{E}_{f_-} \left[ g\left(\frac{f_{+}(\tau)}{f_{-}(\tau)} \right) \right],
\end{align}
\red{where $g(s) = \sqrt{s^2 + 1} - \sqrt{2}$}. 
Equation \eqref{eq.roc.length.0} yields the first important result of this paper:  $\arc{\mathrm{ROC}}(t) - \sqrt{2}$ is an $f$-divergence between score densities $f_+$ and $f_-$ since $g(s)$ is a convex function and $g(1) = 0$. This result confirms, for any given $t$, $\arc{\mathrm{ROC}}(t)$ is a good discrepancy for measuring positive and negative scores (i.e., score distributions). It also explains why the ROC arc length in Figure \ref{fig.ranking.vs.overlapping} is a good discrepancy measure: 
Given $t(x)=x$, the score distributions are same as the data distributions in each plot. 
The arc length of ROCs in Figure \ref{fig.ranking.vs.overlapping} plots are $f$-divergences of the score distributions hence are also $f$-divergences of data distributions.



Although $\arc{\mathrm{ROC}}(t) - \sqrt{2}$ is an $f$-divergence of score distributions, it is not an $f$-divergence of the positive and negative \emph{data} distributions for general $t$. The choice $t(x) = x$ in the toy example does not have simple analogues for higher dimensional datasets. Are there choices of $t$ such that the arc lengths of their ROC curves are good discrepancy measures on data distribution?
In what follows, we show when $t$ is a bijective transform of the likelihood ratio $\frac{p_+(\boldx)}{p_-(\boldx)}$, $\arc{\mathrm{ROC}}(t)$ encodes the differences between $p_+(\boldx)$ and $p_-(\boldx)$ in the form of an $f$-divergence between $p_+(\boldx)$ and $p_-(\boldx)$. 

\subsection{$f$-divergences of Data Distributions}
Using the law of the unconscious statistician, 
we can express $\arc{\mathrm{ROC}}(t)$ in terms of an expectation with respect to the negative data density $p_-(\boldx)$:
$
    \mathbbE_{p_-}\sqrt{\left[\frac{f_{+}(t(\boldx))}{f_{-}(t(\boldx))} \right]^2 + 1}.
$
Consider a special family of score functions: $t^*(\boldx) = \gamma\left(\frac{p_+(\boldx)}{p_-(\boldx)}\right)$ where $\gamma$ is any strictly increasing function.
Due to the Neyman-Pearson lemma \citep{NeymanPearson1933}, $\mathrm{ROC}(t^*)$ has the highest TPR at any FPR level. Geometrically speaking, they dominate all other ROC cuves in an ROC plot and have the maximal AUC. 
Hence, we refer to $t^*$ as the optimal score and  $\mathrm{ROC}(t^*)$ as the optimal ROC curve. 
For convenience, we denote the $\mathrm{ROC}(t^*)$ as $\mathrm{ROC}^*$ which reads ``rock star''. 
It can be shown that 
\begin{align}
    \label{eq.ratiosratioisratio}
    \frac{f_{+}(t^*(\boldx_0))}{f_{-}(t^*(\boldx_0))} =  \frac{\int_{\boldx:  \gamma\left(\frac{p_+(\boldx)}{p_-(\boldx)}\right) = t^*(\boldx_0) }p_+(\boldx) \dx}{\int_{\boldx:  \gamma\left(\frac{p_+(\boldx)}{p_-(\boldx)}\right) = t^*(\boldx_0) } p_-(\boldx)\dx} = \frac{p_+(\boldx_0)}{p_-(\boldx_0)}, ~~~ \forall \gamma
\end{align}
where the second equality holds due to  
$\frac{\int_D a(\boldx)  \mathrm{d}\boldx}{\int_D  b(\boldx) \mathrm{d}\boldx} = \gamma^{-1}(C),$ when $\gamma\left(\frac{a(\boldx)}{b(\boldx)}\right) \equiv C,\ \forall \boldx \in D$. When $\gamma(\boldx) = \boldx$, 
\eqref{eq.ratiosratioisratio} expresses a known result \citep{eguchi2002class}, and is often given in plain English as ``the density ratio of the likelihood ratio score is the likelihood ratio itself''. 

Finally, $\arcrocstar$ takes an elegant form free from $t^*$ or $\gamma$: $    \mathbbE_{p_-(\boldx)} \sqrt{\left[\frac{p_+(\boldx)}{p_-(\boldx)} \right]^2 + 1}$. Equivalently, 
\begin{align}
    \label{eq.optimal.length}
    \arcrocstar - \sqrt{2} = 
    \mathbbE_{p_-(\boldx)} \left[g\left(\frac{p_+(\boldx)}{p_-(\boldx)}\right)\right]
\end{align}

We can see that the same $f$-divergence arises from computing the arc length of $\mathrm{ROC}^*$. However, unlike the $f$-divergence given in \eqref{eq.roc.length.0}, \eqref{eq.optimal.length} is an $f$-divergence between \emph{data} distributions, not score distributions. 
It shows that as long as we use the the optimal scores, the arc length of the optimal ROC can indeed reflect the differences  between \emph{data} distributions. From now on, we will refer to $\arcrocstar - \sqrt{2}$ as the ROC divergence. To the best of our knowledge, \eqref{eq.optimal.length} has not been presented in literature before.

By definition, the ROC divergence is symmetric. Moreover, as a result of Proposition \ref{prop.roc}, the ROC divergence is upper bounded by $2 - \sqrt{2}$ and lower bounded by 0. 
Some geometric properties of $\rockstar$ can be found in Section \ref{proof.convex.longest}. 




%

%

%
%

%
%

%
%
%
%
%

%

%

%

\section{Estimating the Arc Length of $\mathrm{ROC}^*$}
\label{sec.est.arc.length}

\subsection{A Variational Objective}
\label{sec.fench}
%
To numerically approximate the arc length using samples alone, 
we leverage that $\arcrocstar - \sqrt{2}$ is an $f$-divergence.
Utilizing Fenchel's duality \citep{hiriart2004fundamentals}, authors in \cite{Nguyen2010} show that an $f$-divergence $\mathrm{D}_g(p_+|p_-)$ has a variational representation: 
\begin{align*}
    \mathrm{D}_g(p_+|p_-) = \int_\mathcal{X} p_-(\boldx) g\left[\frac{p_+(\boldx)}{p_-(\boldx)}\right] \dx 
    & =  \int_\mathcal{X} p_-(\boldx) \sup_{u} \left\{u(\boldx)\cdot\left[\frac{p_+(\boldx)}{p_-(\boldx)}\right] - g'[u(\boldx)] \right\} \dx \\
    &=  \sup_{u} \int_\mathcal{X} p_+(\boldx) u(\boldx) - \int_\mathcal{X} p_-(\boldx)g'[u(\boldx)] \dx, 
\end{align*}
where $g'$ is the convex conjugate of $g$ and the supremum is taken over all measurable functions. 
In the case of the ROC divergence, $g(z) = \sqrt{z^2 + 1} - \sqrt{2}$ and $z\in[0,\infty]$
thus $g$ has a convex conjugate $g'(z') = -\sqrt{1-z'^2} - \sqrt{2}$, $z' \in [0,1]$. Rewriting $\arcrocstar -\sqrt{2}$ using the above variational representation and dropping the $-\sqrt{2}$, we obtain: 
\begin{align*}
    &\arcrocstar = 
    \sup_{u \in [0,1]} \mathbbE_{p_+}[u(\boldx)]
     + \mathbbE_{p_-} [\sqrt{1 - u^2(\boldx)}].
\end{align*}
We reparameterize $u(\boldx) = \sin[v(\boldx)]$, where $v\in [0,\pi/2]$:
\begin{align}
    \label{eq.roc.length}
    \arcrocstar = 
    \sup_{v \in[0,\pi/2]} \mathbbE_{p_+} \sin[v(\boldx)]  + \mathbbE_{p_-} \cos[v(\boldx)]. 
\end{align}
Differentiating the objective in \eqref{eq.roc.length} for $v$ and setting the derivative to zero, we can see the supremum is attained at $v^* = \atan\frac{p_+}{p_-}$. In other words, the optimal $v^*$ is the arctangent likelihood ratio function. 
\begin{wrapfigure}{r}{0.4\textwidth}
  \begin{center}
    \includegraphics[width=0.4\textwidth]{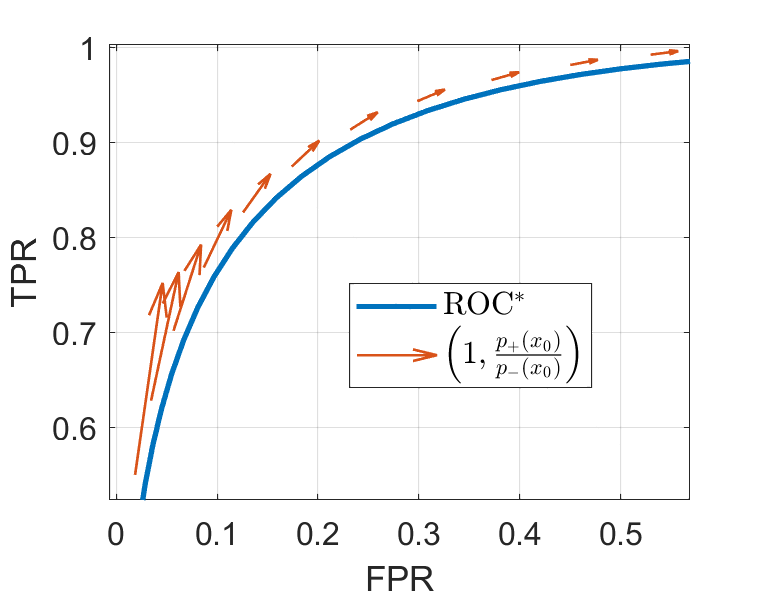}
  \end{center}
\caption{$\rockstar$ and its tangent marked by vector $\left(1, \frac{p_+(\boldx_0)}{p_-(\boldx_0)}\right)$ (scaled to fit). $p_+ = \mathcal{N}(1,1), p_- = \mathcal{N}(-1,1)$.}
\label{fig.roc.atan}
\end{wrapfigure}

It is also interesting to see how $v^*$ is visualized in the ROC plot. 
We can see the tangent of $\rockstar$ at an FPR level $s_0 \in [0,1]$ is 
\begin{align}
\label{eq.slope.angle}
    \partial_s \tilde{F}_+(\tilde{F}_-^{-1}(s_0)) = \frac{f_{+}(\tilde{F}_-^{-1}(s_0))}{f_{-}(\tilde{F}_-^{-1}(s_0))}= \frac{p_+(\boldx_0)}{p_-(\boldx_0)},
\end{align}
where $\boldx_0$ is any point in $\mathcal{X}$ that satisfies the equality $\gamma\left(\frac{p_+(\boldx_0)}{p_-(\boldx_0)}\right)= \tilde{F}_-^{-1}(s_0)$. \eqref{eq.slope.angle} is a known result \citep{flach2016roc}. 
In other words, 
$v^*= \atan \frac{p_+}{p_-} $ is the \emph{slope angle} of $\rockstar$. 
See Figure \ref{fig.roc.atan} for a visualization of the tangent of $\rockstar$ expressed by the likelihood ratio.  



Moreover,  using \eqref{eq.roc.length}, we can obtain a relationship between $\arcrocstar$ and the total variation distance between $\mathbb{P}_+$ and $\mathbb{P}_-$ (denoted as $\mathrm{TV}(\mathbb{P}_+, \mathbb{P}_-)$).
\begin{prop}
    \label{thm:prop.TVbound}
     $\mathrm{TV}(\mathbb{P}_+, \mathbb{P}_-)$ has a lower and upper bound expressed via $\arcrocstar$:
    \red{
    \[ 
    \max_{a\in [0,1]}\frac{2}{\pi}\left[\frac{\arcrocstar - 2\sqrt{1-a^2}}{a} + \arccos(a) - \arcsin(a)\right] \le \mathrm{TV}(\mathbb{P}_+, \mathbb{P}_-) \le \arcrocstar - 1.
    \]}
    \vspace*{-5mm}
\end{prop}
Proof of this proposition can be found in Section \ref{sec:proof.prop.TVBound}. This proposition justifies that $\arcrocstar$ is a valid measure of the discrepancy between $\mathbb{P}_+$ and $\mathbb{P}_-$.  We compare this bound with other known TV bounds in Section \ref{sec.boundcomparisons}.

\subsection{\red{A Tractable Objective for Estimating $\atan{\frac{p_+}{p_-}}$}}
\label{sec.losses}

To use \eqref{eq.roc.length} in practice, we need to find an appropriate function class $\mathcal{F}$.
We can simply restrict $v$ to a bounded (parametric/non-parametric) function class $\mathcal{F}$ and solve the sample version of \eqref{eq.roc.length}:
\begin{align}
    \label{eq.sample.obj}
            \max_{v\in[0,\pi/2], v\in \mathcal{F}} \frac{1}{n_+} \sum_{i=1}^{n_+} \sin(v(\boldx^+_i)) + \frac{1}{n_-} \sum_{i=1}^{n_-} \cos(v(\boldx^-_i)).
\end{align}

In practice, enforcing the boundedness $v \in [0, \pi/2]$ over $\mathcal{X}$ is difficult. We can relax \eqref{eq.sample.obj} by only enforcing the boundedness constraint of $v$ on the sample dataset $X_+ \cup X_-$. 

For example, by letting $\mathcal{F}$ be a Reproducing Kernel Hilbert Space (RKHS) \citep{Scholkopf2001},  we can translate \eqref{eq.sample.obj} into the following optimization problem: 
\begin{align}
    \label{eq:obj_rkhs}
    \hat{v} :=& \argmin_{v\in \mathcal{H}} \ell(v) + \frac{\lambda}{2} \|v\|^2_\mathcal{H}, ~~~ \ell(v) := -\frac{1}{n_+} \sum_{i=1}^{n_+} \sin \langle v, \varphi(\boldx^+_i) \rangle - \frac{1}{n_-} \sum_{i=1}^{n_-} \cos\langle v, \varphi(\boldx^-_i) \rangle \notag \\
    \text{    s.t: } & \langle v, \varphi(\boldx) \rangle \in \left[0, \frac{\pi}{2}\right], ~~~\forall \boldx \in X_+ \cup X_-, 
\end{align}
where $\mathcal{H}$ is a RKHS with a positive definite kernel $k(\boldx, \boldx') = \langle \varphi(\boldx), \varphi(\boldx') \rangle$, $\|\cdot\|_\mathcal{H}$ is the RKHS norm and $\frac{\lambda}{2} \|v\|^2_\mathcal{H}$ is the regularization term.
The optimizer $\langle \hat{v},\varphi(\boldx) \rangle$ is an estimation of $v^*(\boldx)$, the arctangent of the likelihood ratio. 
\eqref{eq:obj_rkhs} is a strictly convex optimization and thus, if a solution $\hat{v}$ exists, it must be unique.

Instead of modelling $\atan \frac{p_+}{p_-}$, we can opt for modelling the log likelihood ratio $\log \frac{p_+}{p_-}$. However, this modelling choice results in an non-convex optimization thus presents extra challenges in the theoretical analysis. 
Details can be found in Section \ref{sec.est.logratio}. 

\subsection{Finite Sample Guarantee}
\label{sec.theory}

%



We show that the solution of \eqref{eq:obj_rkhs}, $\langle \hat{v}, \varphi(\boldx)\rangle$ converges to the true arctangent likelihood ratio (or its projection onto $\mathcal{H}$) as the number of samples goes to infinity.
Below are a few regularity conditions. 
\begin{assumption}
    \label{ass0}
    There exists
    a unique $v^* \in \mathcal{H}$,
    such that $\mathbbE[\nabla_v \ell(v^*)] = 0$ and $\langle v^*, \varphi(\boldx) \rangle \in [0,\pi/2]$ holds for all $\boldx \in \mathcal{X}$. 
    \end{assumption}

A sufficient condition of the above condition is specified in the following proposition. 
\begin{prop}
    \label{prop.optimal.score}
    If there exists a unique $v^* \in \mathcal{H}$, such that $\langle v^*,\varphi(\boldx)\rangle = \atan\left[\frac{p_+(\boldx)}{p_-(\boldx)}\right]$ then Assumption \ref{ass0} holds.
\end{prop}
The proof can be found in Appendix \ref{sec:sec.proof.optimal.score}. Proposition \ref{prop.optimal.score} states if model is correctly specified and identifiable then Assumption \ref{ass0} holds.
It is possible that there exists a $v^* \in \mathcal{H}$ satisfying $\mathbbE[\nabla_v \ell(v^*)] = 0$ which does not meet the boundedness constraint $[0,\pi/2]$. 
In this paper we only consider situations where Assumption \ref{ass0} holds, which includes all situations where the model is correctly specified and some situations where the model is misspecified. 
\begin{assumption}
    \label{ass.local.region}
    Let $n_\mathrm{min} = \min(n_+, n_-)$. 
    There exists a subspace $\mathcal{H}^* :=\{v\in \mathcal{H} | \|v-v^*\|^2_\mathcal{H} \le \delta_{n_\mathrm{min}}^2\}$, such that $\forall v \in \mathcal{H}^*, \forall \boldx \in X_+ \cup X_-, \langle v, \varphi(\boldx) \rangle \in (0, \frac{\pi}{2})$ holds with high probability. 
    The sequence $\delta_{n_\mathrm{min}}$ is monotonically decreasing as $n_\mathrm{min}$ grows to infinity. 
\end{assumption}
Assumption \ref{ass.local.region} states all $v$ within a vicinity of $v^*$ are in the interior of \eqref{eq:obj_rkhs}'s feasible region with high probability. The following proposition gives a sufficient condition under which Assumption \ref{ass.local.region} holds.
\begin{prop}
\label{prop.neigh}
    Suppose $\|\varphi(\boldx)\|_\mathcal{H} \le 1$.
    If our model is correctly specified as described in Proposition \ref{prop.optimal.score} and $\forall \boldx \in \mathcal{X}$, $\atan\left[\frac{p_+(\boldx)}{p_-(\boldx)} \right]\in [R_1, R_2]$, for some $R_1$ and $R_2$ such that $\frac{\pi}{2} >R_2 > R_1>0$, then there exists $N>0$ such that Assumption \ref{ass.local.region} holds when $n_\mathrm{min}>N$. 
\end{prop}
The proof can be found in Appendix \ref{sec.proof.prop.neigh}. 
Since we use RKHS as the estimator function class, our final assumption is that $v^*$ should be reasonably smooth. In previous works such an assumption depends on the decay of the integral operator's eigenvalues \citep{steinwart2009optimal,fukumizu_2009}. 
In this paper, we measure the smoothness using the \emph{range space} technique which has been recently adopted in \citep{fukumizu2013kernel,Sriperumbudur17}. 
We define
\begin{align*}
    \boldSigma_{v} :=  \mathbbE[\nabla^2_v \ell(v)] 
    = \mathbb{E}_{p_+}[\sin \langle v, \varphi(\boldx) \rangle \cdot \varphi(\boldx,\cdot) \otimes \varphi(\boldx,\cdot) ] + 
    \mathbb{E}_{p_-}[ \cos \langle v, \varphi(\boldx) \rangle \cdot \varphi(\boldx,\cdot) \otimes \varphi(\boldx,\cdot)],
\end{align*}
where $\otimes$ denotes the outer product. 
Given $v_0 \in \mathcal{H}$, $\boldSigma_{v_0}$ is an integral operator on $u \in \mathcal{H}$ and
\begin{align*}
    \boldSigma_{v_0}u = \mathbb{E}_{p_+}[\sin \langle v_0, \varphi(\boldx) \rangle  \cdot\varphi(\boldx,\cdot) \cdot u(\boldx) ] + 
    \mathbb{E}_{p_-}[ \cos \langle v_0, \varphi(\boldx) \rangle \cdot \varphi(\boldx,\cdot) \cdot u(\boldx)].
\end{align*}
By definition $\boldSigma_{v_0}$ is a positive, self-adjoint operator, in the sense that $\langle u, \boldSigma_{v_0}u\rangle \ge 0$, $\langle u, \boldSigma_{v_0}v\rangle = \langle \boldSigma_{v_0},u, v\rangle, \forall v,u\in \mathcal{H}$. Moreover, some algebra shows that $\boldSigma_{v_0}$ is also a bounded and compact operator. See Section \ref{sec:proof.operator.properties} for more details. 

Next, we assume the true arctangent ratio function (or its projection) is in the range space of $\boldSigma_{v^*}$. 
\begin{assumption}
    \label{ass.2}
    Let $\mathcal{R}(\boldSigma_{v^*})$ denote the range space of $\boldSigma_{v^*}$.
    There exists $0< \beta \le 1$, $v^* \in \mathcal{R}(\boldSigma^\beta_{v^*})$, where $C ^\beta$ is the fraction power of a compact, positive and self-adjoint operator $C$. 
\end{assumption}
Note that the larger $\beta$ is, the smoother the functions in the range space are. More discussions on the range space assumption can be found in Section 4.2, \citep{Sriperumbudur17}.
Now we are ready to state our theorem: 
\begin{thm}[Convergence Rate of $\hat{v}$] 
    \label{thm:base}
    Suppose Assumptions \ref{ass0}, \ref{ass.local.region} and
    \ref{ass.2} hold and $\hat{v}$ exists. If $\|\varphi(\boldx)\|_\mathcal{H} \le 1$ and 
    \begin{align*}
        \lambda = \frac{T}{n_{\mathrm{min}}^{1/4}}, ~~~
        \frac{K}{n^{
            \beta/4}_\mathrm{min}} \le \delta_{n_\mathrm{min}} \le \frac{4}{{\max\left(B_+, B_-\right)}},
    \end{align*}
    where 
    $
        B_+ = \|\left(\boldSigma_{v^*} + \lambda\boldI\right)^{-1}\mathbbE_{p_+}[\varphi(\boldx)]\|_\mathcal{H}, ~~
        B_- = \|\left(\boldSigma_{v^*} + \lambda\boldI\right)^{-1}\mathbbE_{p_-}[\varphi(\boldx)]\|_\mathcal{H}
    $
    and $T\ge1, K>0$ are constants that do not depend on $n_\mathrm{min}$, then there exists a constant $N>0$ such that when $n_\mathrm{min}>N$, $\|\hat{v} - v^*\|_\mathcal{H} = O_p(n_\mathrm{min}^{-\beta/4}).$
\end{thm}
The proof can be found at Appendix \ref{sec:proof.them.1}.
Theorem \ref{thm:base} shows that, under mild assumptions, $\hat{v}$ is indeed a good estimator for $\atan\left[\frac{p_+(\boldx)}{p_-(\boldx)}\right]$. \red{Some discussions comparing our results with convergence results proved by \citet{Nguyen2008} can be found in Section \ref{sec.comp.conv.results}.  }
Since $\atan\left[\frac{p_+(\boldx)}{p_-(\boldx)}\right]$ is an optimal score that gives rise to $\rockstar$, 
our estimator may have some interesting applications such as outlier detection \cite{Hido2011} or Neyman-Pearson classification \cite{tong13NPclassifier}. We will defer discussions on those applications in future works. 
In the next section, we employ our arctangent likelihood ratio estimator in an application of lower bounding the maximal AUC.

\section{Approximately Lower Bounding the Maximal AUC}
\label{sec.lowerbound.auc}
Finding a score function $t$ that approximately maximizes AUC is an important task in binary classification. 
Let us denote the AUC of $\mathrm{ROC}(t)$ as $\mathrm{AUC}(t)$. It can be seen that \[
\mathrm{AUC}(t) = \int_{[0,1]} \tilde{F}_+(\tilde{F}^{-1}_-(s))\mathrm{d}s =  \mathbbE_{p_-}\mathbbE_{p_+} \left[ \mathbbm{1}\left(t(\boldx_+) \ge t(\boldx_-) \right)\right].\]
Due to the Neyman-Pearson lemma, $\mathrm{ROC}^*$ has the maximum AUC among all ROC curves. Denote the AUC of $\mathrm{ROC}^*$ as $\mathrm{AUC}^*$
Consider the following inequalities: 
\begin{align}
\label{eq.auc.lowerbound}
    \aucstar = \underbrace{\sup_t \mathbbE_{p_-}\mathbbE_{p_+} \left[ \mathbbm{1}\left(t(\boldx^+) \ge t(\boldx^-) \right)\right]}_{(i)} \ge \sup_{t\in \mathcal{F}'} \underbrace{\mathbbE_{p_-} \mathbbE_{p_+} \left[ L\left(t(\boldx^+), t(\boldx^-) \right)\right]}_{(ii)},
\end{align}
where $L(a,b)$ is a continuous and concave lower bound of the indicator function $\mathbbm{1}(a>b)$. 
Due to the Neyman-Pearson lemma,
the supremum of (i) is only attained when $t(\boldx) = \gamma\left(\frac{p_+(\boldx)}{p_-(\boldx)}\right)$ where $\gamma$ is a strictly increasing function. 
Replacing the expectations in (ii) with sample averages yields the optimization problem of AUC maximization \cite{cortes2003,gao13}: 
\begin{align}
\label{eq.auc.max.empirical}
    \max_{t\in \mathcal{F}'} \frac{1}{n_- n_+} \sum_{i=1}^{n_-} \sum_{j=1}^{n_+} L(t(\boldx_j^+), t(\boldx_i^-)).
\end{align}

The objective above is also referred to as Wilcoxon-Mann-Whitney statistic  \citep{hanley1982meaning}.
Therefore, we can see that AUC maximization is a procedure that approximates an optimal score function by maximizing a lower bound of $\aucstar$. 

Now, we show a different way of lower bounding $\aucstar$  
with the help of $\atan{} \frac{p_+(\boldx)}{p_-(\boldx)}$.
We have seen that how $(\tilde{F}_-(\tau), \tilde{F}_+(\tau))$ parameterizes an ROC curve in Section \ref{sec.arclenth}.
In fact, the area between $\rockstar$ and the diagonal line from $(0,0)$ to $(1,1)$ 
can be similarly parameterized by considering ROC curves of positive and negative \emph{mixture  } score distributions. 

\begin{wrapfigure}{r}{0.4\textwidth}
\vspace*{-8mm}
  \begin{center}
    \includegraphics[width=0.4\textwidth]{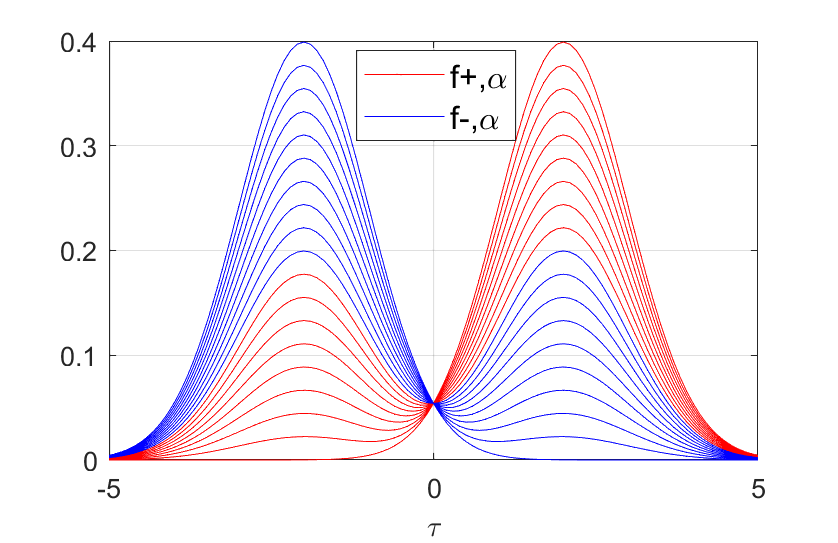}
  \end{center}
\caption{Densities $f^*_+(\tau, \alpha)$ and $f_-^*(\tau, \alpha)$ for $\alpha \in [0,.5]$. }
\vspace*{-5mm}
\label{fig.mixture}
\end{wrapfigure}

Let $F^*_+$ and $F^*_-$ denote CDFs of any optimal score. 
For $\alpha \in [0,.5]$, we can define CDFs of $\alpha$-mixtures of $F^*_+$ and $F^*_-$ as follows
\begin{align*}
    F^*_-(\tau, \alpha) := & (1-\alpha) F^*_-(\tau) + \alpha F^*_+(\tau),\\
    F^*_+(\tau, \alpha) := & \alpha F^*_- (\tau) + (1-\alpha) F^*_+ (\tau).
\end{align*}
Then, FPR ($\tilde{F}^*_-(\tau, \alpha)$) and TPR ($\tilde{F}^*_+(\tau, \alpha)$) for these $\alpha$-mixtures can be defined accordingly. We visualize densities of $F^*_+(\alpha)$ and $F_-^*(\alpha)$ for different $\alpha$ in Figure \ref{fig.mixture}. 

Further, we can see that the 2-dimensional coordinate $\boldr(\tau, \alpha) := (\tilde{F}_-^*(\tau,\alpha), \tilde{F}_+^*(\tau, \alpha))$ parameterizes the area between $\rockstar$ and the diagonal line in $[0,1]^2$:
\begin{itemize}
    \item When fixing $\alpha$ and varying $\tau$, the coordinates give rise to a smooth curve in ROC space from $[0,0]$ to $[1,1]$.
    \begin{itemize}
        \item When $\alpha = 0$, such a curve is $\rockstar$. 
        \item When $\alpha = .5$, such a curve is the diagonal line.
    \end{itemize}
    \item When fixing $\tau = \tau_0$ and varying $\alpha$, the coordinates produce a straight line segment connecting $(\tilde{F}^*_-(\tau_0, 0), \tilde{F}^*_+(\tau_0, 0))$ and $(\tilde{F}^*_-(\tau_0, .5), \tilde{F}^*_+(\tau_0, .5))$.
\end{itemize}
 The left plot in Figure \ref{fig:auc.parameterization} visualizes this parameterization. Now the surface area sandwiched between $\rockstar$ and the diagonal line can be computed using a surface integration:
\begin{figure}
    \centering
     \includegraphics[width=.32\textwidth]{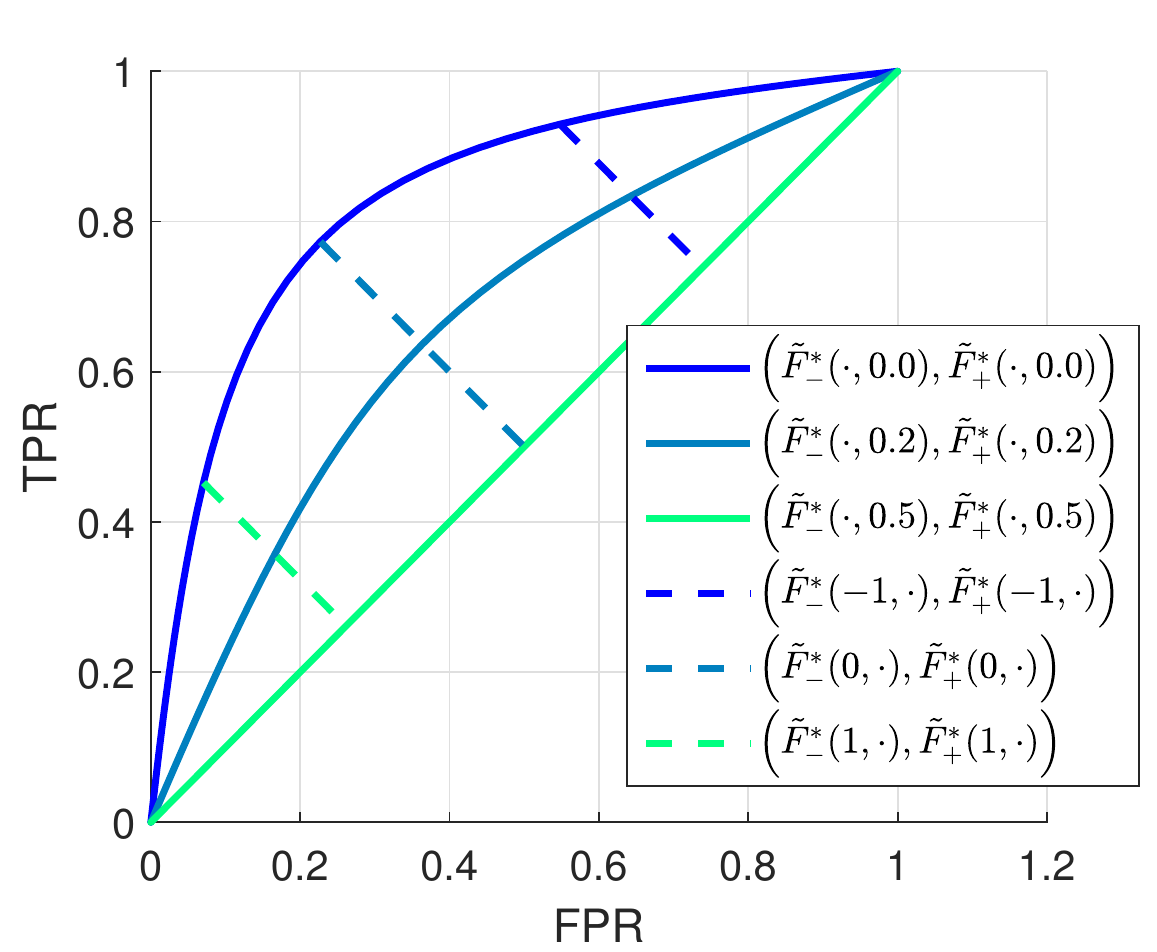}
     \includegraphics[width=.32\textwidth]{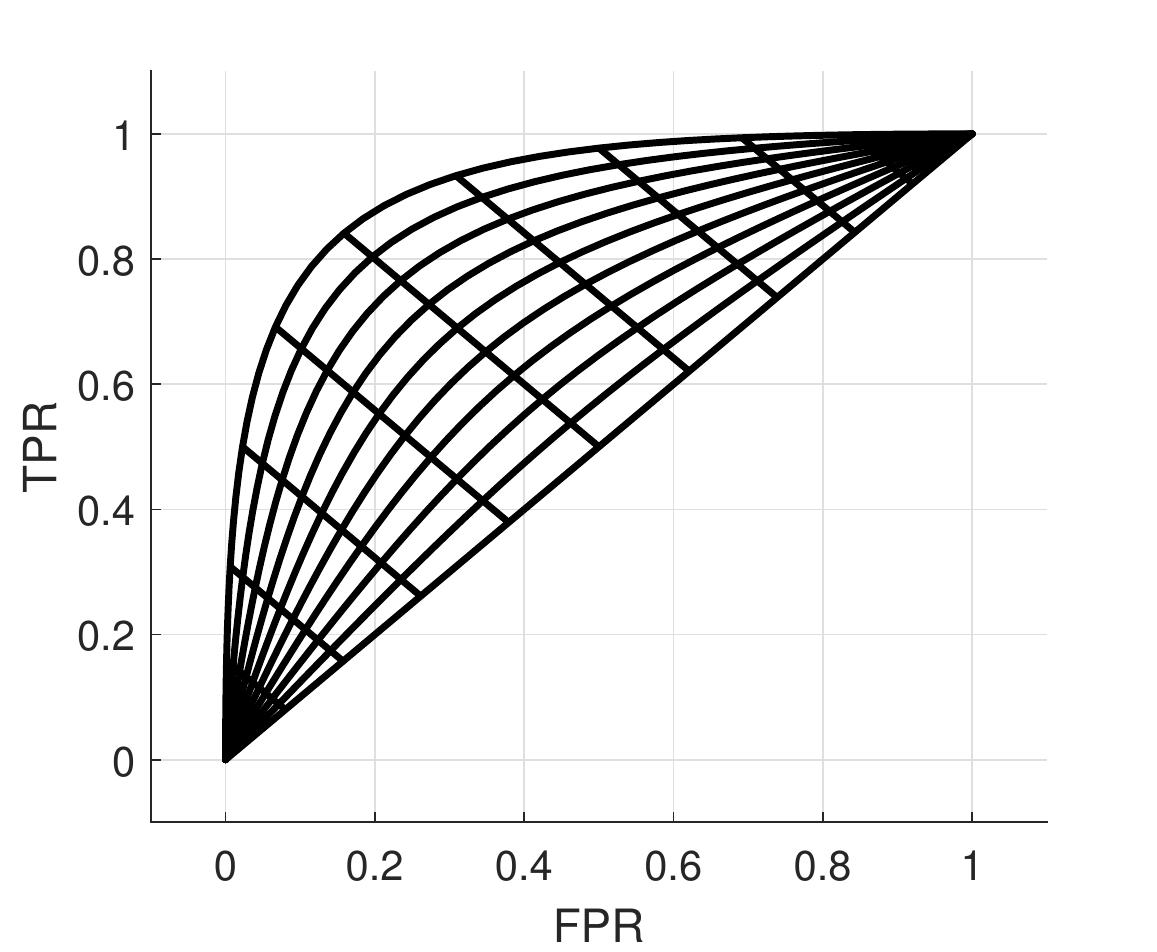}
     \includegraphics[width=.32\textwidth]{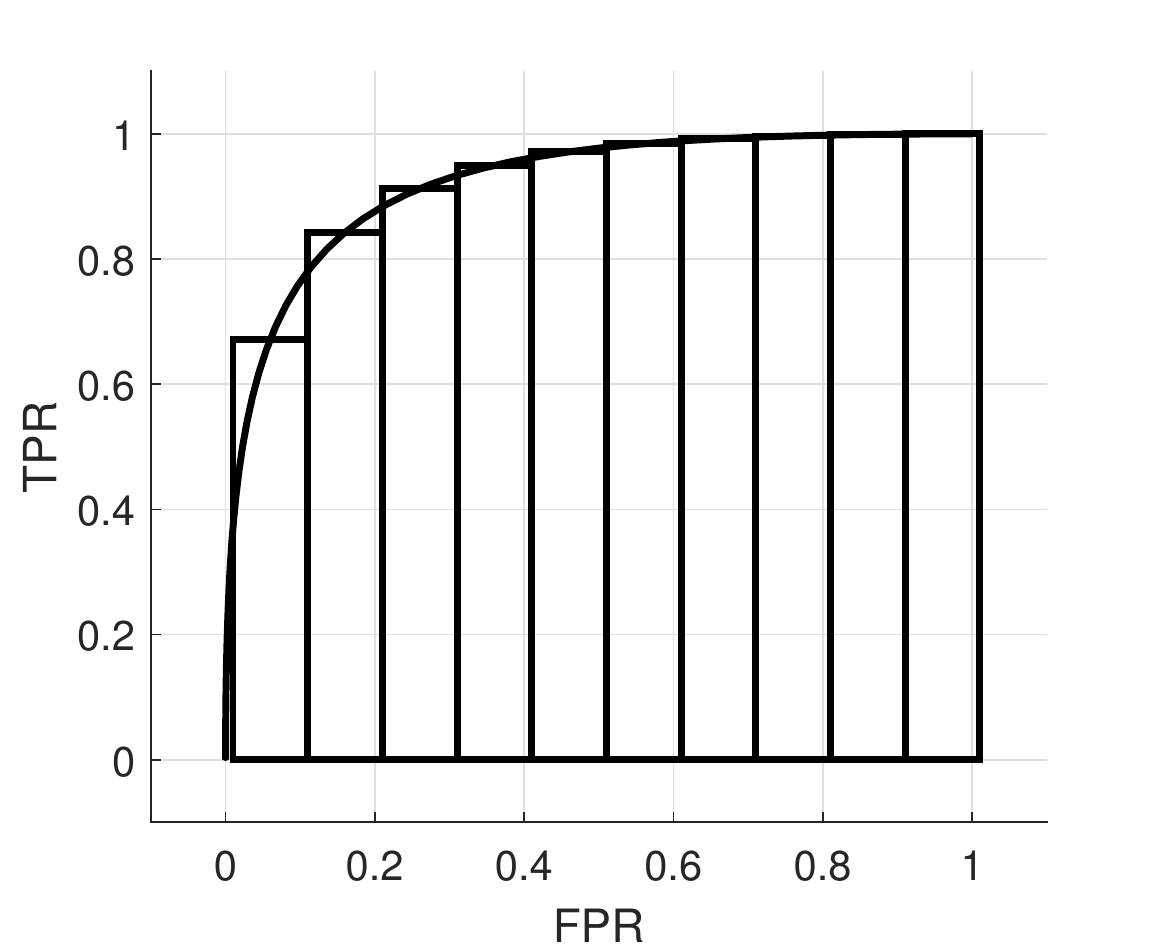}
    \caption{Left: $(\tilde{F}_-^*(\tau,\alpha), \tilde{F}_+^*(\tau, \alpha))$ parameterizes the surface between $\rockstar$ and the diagonal line in $[0,1]^2$. This plot is created by setting $p_+ = \mathcal{N}(1,1)$, $p_- = \mathcal{N}(-1,1)$ and $t^*(x) = \frac{1}{2}\log \frac{p_+(x)}{p_-(x)} = x$. 
    Middle: Our parameterization ``mesh'' divides AUC into surface elements (small patches on the plot). Right:
    Wilcoxon-Mann-Whitney statistic divides AUC into histogram bars. }
    \label{fig:auc.parameterization}
\end{figure}
\begin{align}
\label{eq.auc}
\aucstar - .5 = \int_{\mathrm{dom}(\tau)} \int_{[0,.5]}  \|\partial_\tau \boldr(\tau, \alpha) \times \partial_\alpha \boldr(\tau, \alpha)\|  \ \mathrm{d}\alpha \mathrm{d} \tau,    
\end{align}
    where $\times$ denotes the cross product. 
After some algebra and applying the Fenchel duality technique in Section \ref{sec.fench}, we prove that $\aucstar$ can be expressed as the supremum of a variational objective similar to \eqref{eq.roc.length}:
\begin{prop}
\label{prop.auc.star}
    $\mathrm{AUC}^* = \frac{\sqrt{2} A}{2}+\frac{1}{2}$, 
    \begin{align}
    \label{eq.weighted.arclength}
A := \sup_{v\in [0,\pi/2]} \mathbbE_{p_+} \left[ w\left(\atan{} \frac{p_+(\boldx)}{p_-(\boldx)}\right) \sin[v(\boldx)] \right] + \mathbbE_{p_-} \left[ w\left(\atan{} \frac{p_+(\boldx)}{p_-(\boldx)}\right) \cos[v(\boldx)] \right], 
    \end{align}
    where $w(\tau) :=\sin(\tau + \frac{\pi}{4}) \cdot |F^*_+(\tau) - F^*_-(\tau)|$.
    The supremum of \eqref{eq.weighted.arclength} is attained at $v^* = \atan \frac{p_+}{p_-}$. 
\end{prop}
The proof can be found in Appendix \ref{sec.proof.prop.auc} in the supplementary material. 
A lower bound of $A$ can be obtained by restricting $v$ to a function class. 

Evaluating $w$ requires us to evaluate $\atan \frac{p_+(\boldx)}{p_-(\boldx)}$, $F^*_+$ and $F^*_-$ which are not readily available. However, Section \ref{sec.theory} shows that the empirical estimator \eqref{eq:obj_rkhs} is a consistent estimator of $\atan \frac{p_+(\boldx)}{p_-(\boldx)}$ under mild conditions. Therefore, we propose the following two-step procedure to approximately lowerbound $A$:
\begin{algorithm}
\caption{Two-step Procedure for Approximately Lower Bounding $A$}\label{alg:lowerbound}
\begin{enumerate}
    \item Obtain 
    $\hat{t}(\boldx) := \langle \hat{v},\phi(\boldx)\rangle$  using \eqref{eq:obj_rkhs}.
    Approximate $F^*_+$ and $F^*_-$ using $\hat{F}_+$ and $\hat{F}_-$ which are empirical CDFs of $\hat{t}(\boldx^+)$ and $\hat{t}(\boldx^-)$.
    \item Optimize the empirical version of \eqref{eq.weighted.arclength} by restricting $v$ to a feasible function class and plugging in estimates 
    obtained in the earlier step, i.e.,
    \begin{align}
    \label{eq.auc.maximizer}
        \hathat{v} :=
        \argmax_{v\in [0,\pi/2], v \in \mathcal{F}} \frac{1}{n_+} \sum_{i=1}^{n_+}  \hat{w}\left[ \hat{t}(\boldx_i^+)\right] \cdot \sin[v(\boldx^+_i)] 
        + \frac{1}{n_-} \sum_{i=1}^{n_-} \hat{w}\left[\hat{t}(\boldx_i^-)\right] \cdot \cos[v(\boldx^-_i)],
    \end{align}
    where $\hat{w}(\tau) := \sin(\tau + \frac{\pi}{4}) \left|\hat{F}_+(\tau) - \hat{F}_-(\tau)\right|$.
\end{enumerate}
\end{algorithm}

Note that \eqref{eq.auc.maximizer} is nothing but a weighted sample objective \eqref{eq.sample.obj}. Thus, it can be easily optimized by the algorithm that solves a weighted version of \eqref{eq.sample.obj} given the approximated weights in the first step. In practice, we simply run the solver for \eqref{eq:obj_rkhs} twice: The first time we run it without weights then run it again with weights $\hat{w}\left[ \hat{t}(\boldx_i)\right]$ calculated from the first run. 

Since the above algorithm also approximates an optimal score ($\atan \frac{p_+}{p_-}$) by maximizing an approximated lower bound of $\aucstar$, it is natural to wonder how the maximizer
$\hathat{v}$
 of \eqref{eq.auc.maximizer} would perform in AUC maximization tasks. In the next section, we show that our two-step algorithm achieves a promising AUC performance compared to a state of the art AUC maximizer. 

Computing $\aucstar$ using \eqref{eq.auc} is different from using Wilcoxon-Mann-Whitney statistic (i.e., \eqref{eq.auc.maximizer}): Our approach divides the space between $\rockstar$ and the diagonal into small surface elements and then adds them up. Wilcoxon-Mann-Whitney statistic adds up all histogram bars, which are TPRs at different FPR levels. Our approach requires $\tilde{F}_+$ and $\tilde{F}_-$ to be differentiable with respect to $\tau$, which means the score distributions cannot be discrete. 
However, Wilcoxon-Mann-Whitney can compute the AUC of discrete score distributions without a problem. This difference is visualized in the middle and right plots of Figure \ref{fig:auc.parameterization}. 


\section{Experiments}
\subsection{Numerical Comparison of  Divergences and $\mathrm{TV}$ Bounds}
\label{sec.boundcomparisons}
\begin{wrapfigure}{r}{0.4\textwidth}
\centering
  \centering
  \includegraphics[width=.97\linewidth]{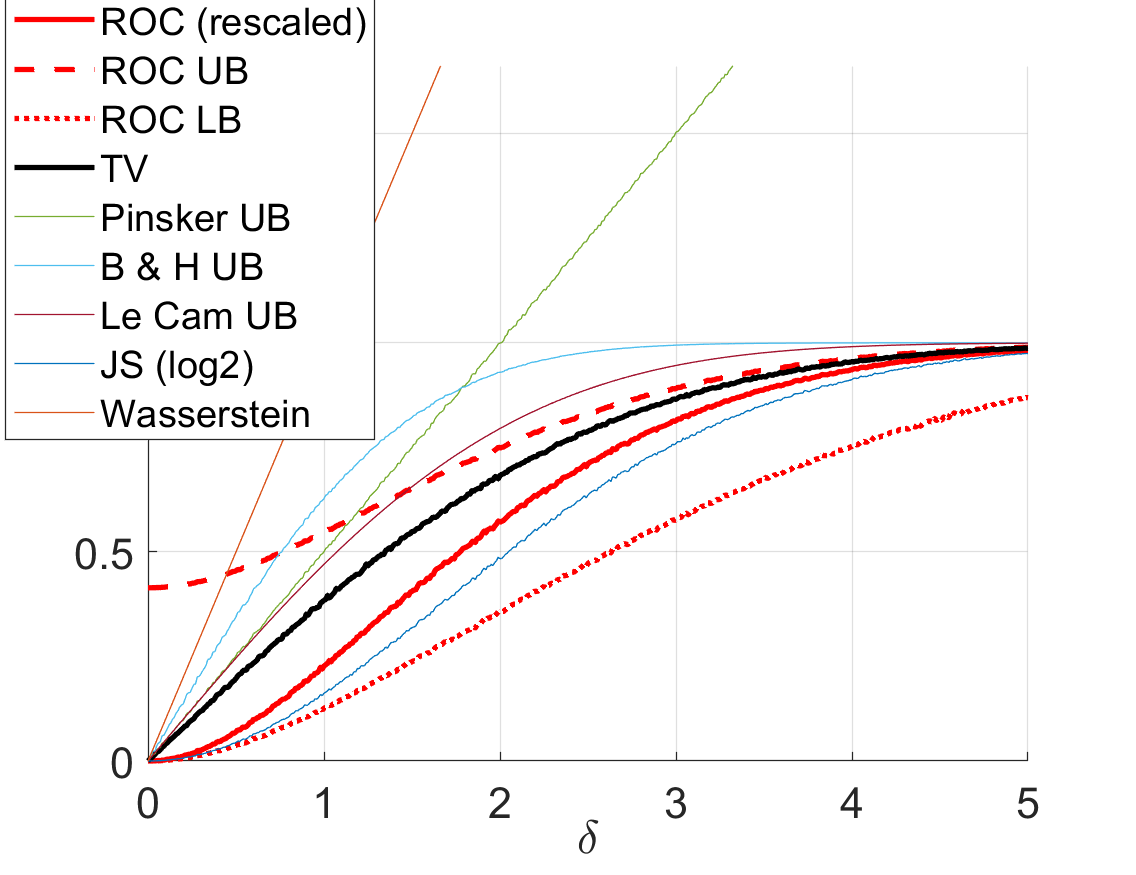}
  \captionof{figure}{\red{Comparison of various divergences and bounds of $\mathrm{TV}$.
  }}
  \label{fig:bound}
\end{wrapfigure}

\red{In this experiment, we numerically compare the ROC divergence, the upper and lower bound in Proposition \ref{thm:prop.TVbound} with several other divergences and some known bounds of $\mathrm{TV}$ in Figure \ref{fig:bound}. In this numerical simulation, $p_+ = \mathcal{N}(0,1)$ and $p_- = \mathcal{N}(\delta,1)$. 
We plot ROC divergence, Jensen Shannon divergence, Wasserstein distance and $\mathrm{TV}$ between $p_+$ and $p_-$ as $\delta$ grows from 0 to 5. We can see the (rescaled) ROC divergence closely resembles $\mathrm{TV}$. When $\delta > 1.45$, the upper bound given in Proposition \ref{thm:prop.TVbound} is the tightest among known TV upper bounds \citep{canonne2022short,devroye2018total} (Pinsker's upperbound,  Bretagnolle \& Huber's upper bound, Le Cam's upper bound). This suggests that combining our upperbound with existing bounds may produce an even tighter bound for $\mathrm{TV}$. } 

\subsection{Imbalanced Classification on CIFAR-10}
\label{sec.exp}
\begin{figure}[t]
    \centering
    \includegraphics[width=.99\textwidth]{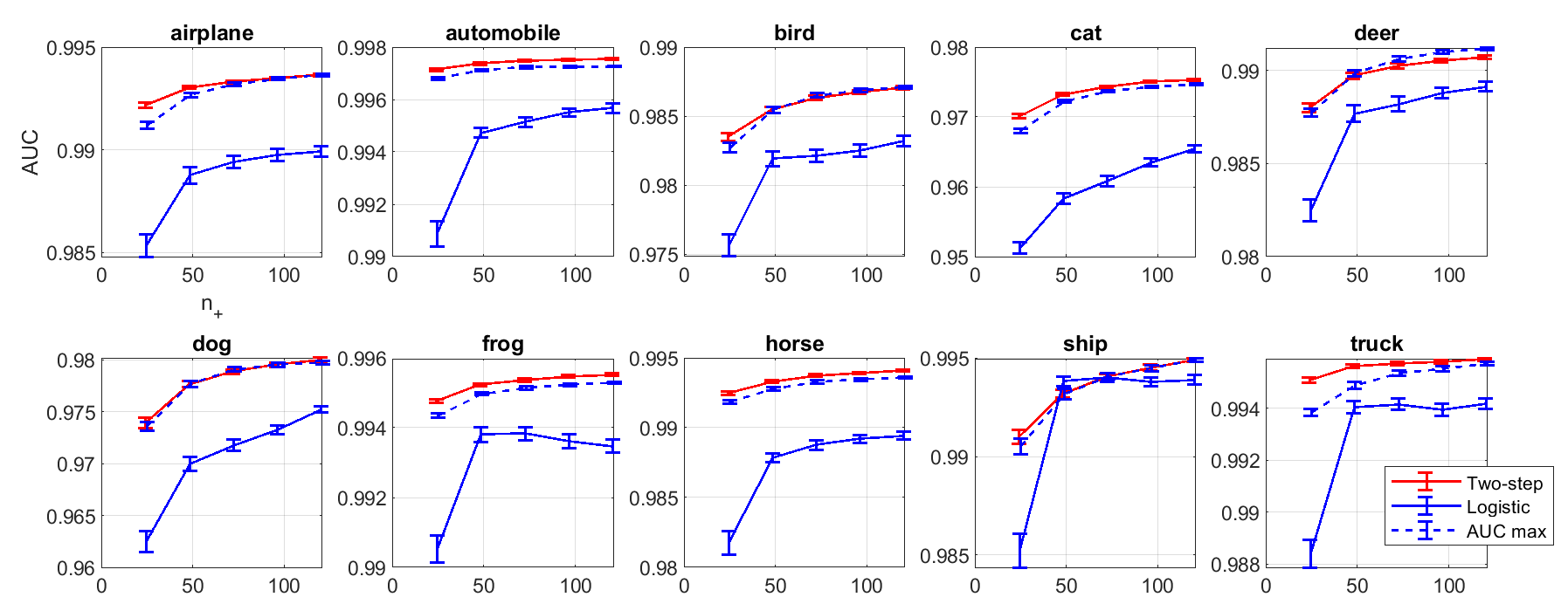}
    \caption{Testing AUC of one-versus-the-rest classification on CIFAR-10 dataset.}
    \label{fig:cifar10}
\end{figure}

In this section, we test if the $\hathat{v}$ obtained in our two-step procedure \eqref{eq.auc.maximizer} is indeed a good score function in terms of AUC in imbalanced classification tasks. We use a widely known image classification dataset CIFAR-10 \cite{krizhevsky2009learning}.   
The performance is compared with an AUC maximizer which maximizes the empirical lower bound in \eqref{eq.auc.max.empirical} and a vanilla logistic regression classifier. We set \red{the surrogate loss} $L(a,b):= -(1-(a-b))^2$ in the AUC maximizer, as suggested in \cite{gao13}. All methods use linear models with no regularization terms since our models are simple and we have sufficient samples. Particularly, the $\hat{t}$ and $\hathat{v}$ in our two-step algorithm is obtained using \eqref{eq:obj_rkhs} by setting $\varphi(\boldx) = \boldx$ and $\lambda = 0$. \red{The AUC maximization (AUC max) is implemented using SPAUC method \cite{lei2021stochastic}}. 

Instead of using the raw features, we extract 50 dimensional bounded features by training a residual network \cite{he2016deep} on the training dataset using the 10-class cross entropy loss. The structure of the network is included in the supplementary material. 
After obtaining features, we construct datasets for 10 different one-versus-the-rest classification tasks. 
For a single task, we pick a class and 
obtain $X_+$ by randomly sampling from this  class $n_+$ times in the training set. Similarly, $X_-$ is obtained by randomly sampling from the rest of the classes $n_-$ times. 
In our experiments, we set $n_+ = 24, 48, 72, 96, 120$ and fix $n_- = 1000$ to create imbalanced positive and negative datasets. We run all three methods and obtain the corresponding score functions. For each class, we repeat the experiment 96 times using different random samples. We use the testing and training split provided by the dataset itself. 

Our experiments can be seen as a transfer learning task which reuses predictive features trained for a multi-class classifier for one-versus-the-rest binary classification tasks. 

The average AUCs computed on the testing dataset and their standard errors over 96 runs over different $n_+$ sample sizes are shown in Figure \ref{fig:cifar10}. Our method has approximately equal performance with the AUC maximizer despite not directly maximizing the AUC. This observation indicates that $\hathat{v}$ can be a good score function in AUC maximization tasks. Both of the methods significantly outperform vanilla logistic regression. 

\subsection{Discussions on Computational Complexity}
\label{sec.comp.complexity}
Without loss of generality, assume $n_+ = n_- = n$. The naive caclulation of objective \eqref{eq.auc.max.empirical} has a computational complexity $O(n^2)$ since we evaluate the loss function $L$ at each pair of samples. 
However, authors in \cite{joachims2005support} have shown that the objective function in \eqref{eq.auc.max.empirical} can also be computed with $O(n\log(n))$ complexity for hinge loss (and decomposable loss functions). 
A recent work \cite{yingSOAUC} simplifies the computation of \eqref{eq.auc.max.empirical} for the squared loss function $L$ with an unbiased estimate. Suppose $t(\boldx) := \langle v, \boldx \rangle$, then the \emph{negative} objective of \eqref{eq.auc.max.empirical} is an unbiased estimate of
\begin{align}
\label{eq.ying}
    1 + & \mathrm{Var}_{p_+}[\langle v, \boldx \rangle] + \mathrm{Var}_{p_-}[\langle v, \boldx \rangle] + 2 \langle v, \mathbbE_{p_-}[\boldx] -  \mathbbE_{p_+}[\boldx] \rangle + \langle v, \mathbbE_{p_-}[\boldx] -  \mathbbE_{p_+}[\boldx] \rangle^2.
\end{align}
After we approximating \eqref{eq.ying} with empirical terms, the computation can be done with a complexity $O(n)$. When implemented in an online fashion, it has a computational complexity of one datum.  

\red{In comparison, the objective \eqref{eq.sample.obj} and \eqref{eq.auc.maximizer} are summation of $\sin/\cos(v(\boldx))$ evaluated at each datum, so computing the objective/gradient has a computational complexity $O(n)$. } 
Computing $\hat{F}_+$ and $\hat{F}_-$ requires sorting our dataset, which has an average complexity $O(n\log n)$. However, once our datasets are sorted, 
$\hat{F}_+\left(\hat{t}(\boldx_0)\right) = \frac{i}{n_+}$, where $i$ is the index of $\hat{t}(\boldx_0)$ in the sorted set $\{\hat{t}(\boldx_i)\}_{i=1}^{n_+}$. 

\section{Conclusions}
In this paper, we show that a novel $f$-divergence arises from the arc length of the optimal ROC curve. The arc length can be accurately estimated from positive and negative samples using a variational expression. It is also an estimator for $\atan{p_+/p_-}$ and has a convergence rate $O_p(n^{-\beta/4})$.  Finally, we show that the area between the optimal ROC curve and the diagonal can be parameterized  using a similar variational objective. It leads to a two-step procedure that approximately lower bounds the maximal AUC which achieves a promising result in AUC maximization tasks. 

\begin{ack}
The author would like to thank Prof. Peter Flach and Dr. Hao Song for their helpful discussions. 
The author would like to thank four anonymous reviewers for their insightful comments. In particular, we thank Reviewer WSBr and the Area Chair for pointing out the computational complexity inaccuracies in our initial version. 
\end{ack}

\bibliographystyle{plainnat}
\bibliography{main}

\section*{Checklist}

\begin{enumerate}

\item For all authors...
\begin{enumerate}
  \item Do the main claims made in the abstract and introduction accurately reflect the paper's contributions and scope?
  
    \answerYes{In introduction, we discussed the background/usages ROC curves, $f$-divergence estimation and provided a full paragraph summary of the contribution of this paper. We give the same contribution rundown in the abstract as well. }
  \item Did you describe the limitations of your work?
  
    \answerYes{In Section \ref{sec.lowerbound.auc}, we included a brief discussion comparing the proposed maximal AUC lower bounding approach vs. the classic Wilcoxon-Mann-Whitney statistic.}
  \item Did you discuss any potential negative societal impacts of your work?
    \answerNA{}
  \item Have you read the ethics review guidelines and ensured that your paper conforms to them?
  
    \answerYes{}
\end{enumerate}

\item If you are including theoretical results...
\begin{enumerate}
  \item Did you state the full set of assumptions of all theoretical results?
  
    \answerYes{See Section \ref{sec.theory}}
  \item Did you include complete proofs of all theoretical results?
  
    \answerYes{See Appendix}
\end{enumerate}

\item If you ran experiments...
\begin{enumerate}
  \item Did you include the code, data, and instructions needed to reproduce the main experimental results (either in the supplemental material or as a URL)?
  
    \answerYes{Full instructions on how to reproduce our experiments is provided in the supplementary material.}
  \item Did you specify all the training details (e.g., data splits, hyperparameters, how they were chosen)?
  
    \answerYes{See Section \ref{sec.exp}}
        \item Did you report error bars (e.g., with respect to the random seed after running experiments multiple times)?
  
    \answerYes{See Section \ref{sec.exp}}
        \item Did you include the total amount of compute and the type of resources used (e.g., type of GPUs, internal cluster, or cloud provider)?
    
    \answerNA{The Computation time is not compared in this paper. }
\end{enumerate}

\item If you are using existing assets (e.g., code, data, models) or curating/releasing new assets...
\begin{enumerate}
  \item If your work uses existing assets, did you cite the creators?

    \answerYes{See Section \ref{sec.exp}}
  \item Did you mention the license of the assets?
  
    \answerNo{The authors of CIFAR-10 do not specify a license for the dataset.}
  \item Did you include any new assets either in the supplemental material or as a URL?
  
    \answerNo{We did not create any new asset in this research. }
  \item Did you discuss whether and how consent was obtained from people whose data you're using/curating?
  
    \answerNo{We did not collect any data ourselves.}
  \item Did you discuss whether the data you are using/curating contains personally identifiable information or offensive content?
  
    \answerNA{CIFAR-10 dataset we use in our experiment does not contain any identifiable human information.}
\end{enumerate}

\item If you used crowdsourcing or conducted research with human subjects...
\begin{enumerate}
  \item Did you include the full text of instructions given to participants and screenshots, if applicable?
  
    \answerNA{}
  \item Did you describe any potential participant risks, with links to Institutional Review Board (IRB) approvals, if applicable?
  
    \answerNA{}
  \item Did you include the estimated hourly wage paid to participants and the total amount spent on participant compensation?
  
    \answerNA{}
\end{enumerate}

\end{enumerate}


\newpage
\appendix


\section{Proof of Proposition \ref{prop.roc}}
\label{sec.poof.prop.roc}
\begin{proof}
    Jensen's inequality: 
    \begin{align*}
        \arc{\mathrm{ROC}}(t) = \sqrt{2} \int \sqrt{\frac{f^2_{+}(\tau)}{2} + \frac{f^2_{-}(\tau)}{2}} \mathrm{d} \tau \ge  \sqrt{2} \int \frac{|f_{+}(\tau)|}{2} + \frac{|f_{-}(\tau)|}{2} \mathrm{d} \tau 
        = \sqrt{2}.
    \end{align*}
    Triangle inequality: 
    $
        \arc{\mathrm{ROC}} \le \int |f_{+}(\tau)| + |f_{-}(\tau)|  \mathrm{d} \tau  
        = 2.
    $
\end{proof}

\section{Geometric Properties of $\rockstar$}
\label{proof.convex.longest}
Here we prove a result regarding some other geometric properties of $\rockstar$.
\begin{prop}
    \label{prop.convex.longest}
    $\mathrm{ROC}^*$ is a convex curve and $\arcrocstar$ is the longest among all convex ROC curves. 
\end{prop}

\begin{proof}
    First, we show $\rockstar$ is a convex curve.
    To show $\rockstar$ is convex, we only need to show $\tilde{F}_+(\tilde{F}^{-1}_-(s))$ is a concave function. 
    This can be verified by checking the sign of $\partial^2_s\tilde{F}_+(\tilde{F}_-^{-1}(s))$:
    \begin{align*}
        \partial_s \tilde{F}_+(\tilde{F}_-^{-1}(s)) &= 
        \frac{ f_{+}[\tilde{F}_-^{-1}(s)]}{f_{-}[\tilde{F}_-^{-1}(s)]} 
        = \frac{p_+(\boldx_0)}{p_-(\boldx_0)} = \gamma^{-1}\left(\tilde{F}_-^{-1}(s)\right),
    \end{align*}
    where the second equality is due to \eqref{eq.ratiosratioisratio} and $\boldx_0$ is any point in $\mathcal{X}$ that satisfies the equality $\gamma\left(\frac{p_+(\boldx_0)}{p_-(\boldx_0)}\right)= \tilde{F}_-^{-1}(s)$. Further, we can show that, 
    \begin{align*}
        \partial^2_s \tilde{F}_+(\tilde{F}^{-1}_-(s;t^*)) &= 
        -\frac{1}{\partial_s \gamma[\gamma^{-1}(\tilde{F}^{-1}_-(s))]} \cdot \frac{1}{f_-[(\tilde{F}^{-1}_-(s)]}. 
    \end{align*}
    Since $\gamma$ is a strictly monotone increasing function, the first factor is non-negative and the second factor is also strictly positive due to the our assumption on the positivity of $f_-$. We have $\partial^2_s F_+(F^{-1}_-(s;t^*))\le 0$. 
    Moreover, at any FPR level $s\in [0,1]$, the Neyman-Pearson lemma \cite{NeymanPearson1933} implies
    $$
        \tilde{F}_+(\tilde{F}_-^{-1}(s)) \ge \tilde{F}'_+(\tilde{F}_-^{'-1}(s)), 
    $$
    where $\tilde{F}'_+$ and $\tilde{F}'_-$ are TPR and FPR of any other score function. 
    In words, $\rockstar$ dominates all other ROC curves. 
    Since $\rockstar$ is convex and encloses all other ROC curves, our claim follows Archimedes's Second Axiom: among all convex curves with the same endpoints, the one encloses all other curves has the longest arc length. 
\end{proof}

\section{Proof of Proposition \ref{thm:prop.TVbound}}
\label{sec:proof.prop.TVBound}

\begin{proof} 
    Using the integral probability metric representation of $\mathrm{TV}(\mathbb{P}_+, \mathbb{P}_-)$ \citep{sriperumbudur2012empirical}, we can write:  
    \begin{align*}
        \frac{\pi}{2}\mathrm{TV}(\mathbb{P}_+, \mathbb{P}_-)  
        =& \sup_{\|v\|_\infty \le 1} \mathbbE_{p_+} \left[\frac{\pi}{2} \cdot \frac{(v(\boldx) + 1)}{2}\right] - \mathbbE_{p_-} \left[\frac{\pi}{2} \cdot \frac{(v(\boldx) + 1)}{2} \right]\\
        =& \sup_{v'\in [0,\pi/2]} \mathbbE_{p_+} [v'(\boldx)]  + \mathbbE_{p_-} \left[ - v'(\boldx) \right]
    \end{align*}
    \red{
    Some algebra can show that $z \ge \frac{\sin(z)}{a}  + \arccos(a) - \frac{\sqrt{1-a^2}}{a} $ and $-z \ge \frac{\cos(z)}{a}  - \arcsin(a) - \frac{\sqrt{1-a^2}}{a} $ for all $a \in [0,1]$ and $z \in [0, \pi/2]$. Therefore 
    \begin{align*}
        \frac{\pi}{2}\mathrm{TV}(\mathbb{P}_+, \mathbb{P}_-)  
        &\ge \sup_{v'\in [0,\pi/2]} \mathbbE_{p_+} \left[\frac{\sin(\boldx)}{a}\right]  + \mathbbE_{p_-} \left[  \frac{\cos(\boldx)}{a} \right] + \arccos(a) - \arcsin(a) - \frac{2\sqrt{1-a^2}}{a}
        \\
        &\ge \frac{\arcrocstar}{a} + \arccos{a} - \arcsin{a} - \frac{2\sqrt{1-a^2}}{a}.
    \end{align*}
    }
    Similarly, multiplying both sides of the second equality above by $\frac{2}{\pi}$, we obtain 
    \begin{align*}
        \mathrm{TV}(\mathbb{P}_+, \mathbb{P}_-)  
        &= \sup_{v'\in [0,\pi/2]} \mathbbE_{p_+} \frac{2}{\pi}v'(\boldx)  + \mathbbE_{p_-} \left[ - \frac{2}{\pi}v'(\boldx) \right]\\
        &\le \sup_{v'\in [0,\pi/2]} \mathbbE_{p_+}  \sin(v'(\boldx))+ \mathbbE_{p_-} \left[ \cos v'(\boldx) -1 \right] \\
        &= \arcrocstar - 1.
    \end{align*}
\end{proof}

\section{Proof of Proposition \ref{prop.neigh}}
\label{sec.proof.prop.neigh}
\begin{proof}
      $\forall v\in \mathcal{H}^*, |\langle v,\varphi(\boldx) \rangle - \langle v^*,\varphi(\boldx) \rangle| \le \|v-v^*\|_\mathcal{H}\|\varphi(\boldx)\|_\mathcal{H}\le \delta_{n_\mathrm{min}}$. If $\delta_{n_\mathrm{min}} < \min(R_1, \frac{\pi}{2}-R_2)$ then $\langle v, \varphi(\boldx) \rangle \in (0, \frac{\pi}{2})$ holds uniformly for every $\boldx \in \mathcal{X}$. As $\delta_{n_\mathrm{min}}$ is a decaying sequence, there always exists an $N$ such that  $\delta_{n_\mathrm{min}} \le \min(R_1, \frac{\pi}{2}-R_2)$ holds for $n_\mathrm{min}\ge N$.     %
\end{proof}

\section{Proof of Theorem \ref{thm:base}}
\label{sec:proof.them.1}

To reduce the visual clutter, in this section, $\|v\|$ represents the Hilbert space norm of $v$, defined as $\sqrt{\langle v, v \rangle}$. We simplify  $\mathbbE_{p_+}  [v(\boldx)]$ as $\mathbbE_+[v(\boldx)]$ whenever it does not lead to confusion. For ease, we write $\sum_{i=1}^{n_+} f(\boldx^+_i)$ as $\sum_{i=1}^{n_+} f(\boldx_i)$, a convention which will adopted henceforth.

\begin{proof}
    Define $\mathcal{H}^* := \{v\in \mathcal{H}| \|v-v^*\|^2\le \delta^2\}$. 
    Consider an optimization that is similar to \eqref{eq:obj_rkhs}: 
    \begin{align}
        \label{eq.obj.cons}
        \tilde{v} &:= \argmin_{v\in\mathcal{H}^*} \ell(v) + \frac{\lambda}{2}\|v\|^2
    \end{align} 

    Define $\tilde{u} := \tilde{v} - v^*$ and we have the following equality due to the KKT conditions of \eqref{eq.obj.cons} 
    \begin{align*}
    \nabla_v \ell(\tilde{v}) + \lambda \tilde{v} + 2 \nu \tilde{u} = 0, 
   \end{align*}
   where $\nu$ is a Lagrangian multiplier and $\nu \ge 0$. Multiplying both sides by $\tilde{s} = \left(\boldSigma_{v^*}+\lambda\boldI\right)^{-1}\tilde{u}$, we have 
   \begin{align*}
    \langle \tilde{s}, \nabla_v \ell(\tilde{v}) + \lambda \tilde{v} + 2 \nu \tilde{u} \rangle = 0.
   \end{align*}

    Let $g(v) := \langle \tilde{s}, \nabla_v \ell(v) + \lambda v +  2\nu (v - v^*) \rangle, $ we can 
    applying Mean Value Theorem (MVT) on the scalar valued function $g(v)$:
    \begin{align}
    \label{eq.mvt}
        g(\tilde{v}) - g(v^*) = \nabla_v g(\bar{v})\tilde{u}, 
    \end{align}
    where $\bar{v} = a v^* + (1-a)\hat{v}$ for some $a\in [0,1]$.
    Knowing $g(\tilde{v}) = 0$ and $g(v^*) = \langle \tilde{s}, \nabla_v \ell(v^*) +\lambda v^* \rangle$, we can translate \eqref{eq.mvt} into 
    \begin{align}
    \label{eq.proof.kkt}
        \langle \tilde{s}, -\nabla_v \ell(v^*) - \lambda v^*  \rangle = \langle \tilde{s}, [\nabla^2_f \ell(\bar{v}) + \lambda\boldI + 2\tilde{\nu}\boldI] \tilde{u}\rangle, 
    \end{align}
    where $\boldI$ is the identify matrix. 
    Focusing on the RHS, 
    we have
    \begin{align}
        \label{eq.proof.0}
        \langle \tilde{s}, [\nabla^2_f \ell(\bar{v}) + \lambda\boldI + 2\nu \boldI] \tilde{u}\rangle 
        \ge&\langle \tilde{s}, [\nabla^2_f \ell(\bar{v}) + \lambda\boldI] \tilde{u} \rangle \notag \\
        \ge&  \underbrace{\langle (\boldSigma_{v^*} + \lambda\boldI)^{-1}\tilde{u},[\boldSigma_{v^*} + \lambda\boldI ] \tilde{u}\rangle }_{\|\tilde{u}\|^2} - \textcolor{red}{\underbrace{\langle \tilde{s}, [\boldSigma_{v^*} - \boldSigma_{\bar{v}}] \tilde{u}\rangle }_{a}}  - \textcolor{blue}{\underbrace{ \langle \tilde{s}, [ \boldSigma_{\bar{v}} - \nabla^2_f \ell(\bar{v})  ] \tilde{u}\rangle}_{b}} \notag \\
        \ge & \|\tilde{u}\|^2-\textcolor{red}{a}-\textcolor{blue}{b}.
    \end{align}
    The first line is due to the fact that $\langle 2\nu\tilde{s}, \tilde{u}\rangle \ge 0$. 
    Use the inequality \eqref{eq.proof.0} on \eqref{eq.proof.kkt}, we get the inequality 
    \begin{align}
        \label{eq.proof.kkt2}
        \langle \tilde{s}, -\nabla_v \ell(v^*) - \lambda v^* \rangle  \ge \|\tilde{u}\|^2 - a - b.
    \end{align}
    First, let us inspect \textcolor{red}{$a$}. Using MVT on $\sin\langle v, \varphi(\boldx)\rangle, v\in \mathcal{H}^*$ and applying Hölder's inequality, we get
    \begin{align}
        \label{eq.sin.lip}
        \sin\langle v^*, \varphi(\boldx) \rangle - \sin\langle v^*+\bolddelta', \varphi(\boldx) \rangle \le  \|\bolddelta'\| \cdot \|\varphi(\boldx)\|\le  \delta \cdot \|\varphi(\boldx)\|.
    \end{align}
    Define $\boldSigma^{+}_v := \mathbb{E}_+ \left[\sin \langle v,\varphi(\boldx) \rangle \varphi(\boldx) \otimes \varphi(\boldx)\right]$
    and $\hat{\boldSigma}^{+}_v$ as its empirical 
    counterpart approximated using $X_+$. 
    We can see that $a = \langle \tilde{s}, [\boldSigma^+_{v^*} - \boldSigma^+_{\bar{v}}]\tilde{u} \rangle + \langle \tilde{s}, [\boldSigma^-_{v^*} - \boldSigma^-_{\bar{v}}] \tilde{u} \rangle$. Moreover, 
    \begin{align*}
        \langle \tilde{s}, [\boldSigma^+_{v^*} - \boldSigma^+_{\bar{v}}]\tilde{u} \rangle &\stackrel{\mathrm{i}}{\le} \mathbbE_+ \left\{\delta\cdot\|\varphi(\boldx)\| \cdot \langle \tilde{s}, \varphi(\boldx) \otimes \varphi(\boldx) \tilde{u}\rangle \right\}  \\
        & \le \delta \langle \tilde{u}, \mathbbE_+ \left\{ (\boldSigma_{v^*} + \lambda\boldI)^{-1} \varphi(\boldx) \rangle \cdot \|\varphi(\boldx)\| \right\} \cdot\|\tilde{u}\|\\
        & \le \delta \|\tilde{u}\| \cdot \| (\boldSigma_{v^*} + \lambda\boldI)^{-1} \mathbbE_+ \varphi(\boldx) \| \cdot \|\tilde{u}\| 
    \end{align*}
    ($\mathrm{i}$) is due to \eqref{eq.sin.lip}. 
    Following a similar line of reasoning, we can see
    \begin{align*}
        \langle \tilde{s}, [\boldSigma^-_{v^*} - \boldSigma^-_{\bar{v}}] \tilde{u} \rangle \le \delta \|\left(\boldSigma_{v^*} + \lambda\boldI\right)^{-1}\mathbbE_-[\varphi(\boldx)]\|\cdot \|\tilde{u}\|^2.
    \end{align*}
    By setting $\delta \le {{4\max\left(\|\left(\boldSigma_{v^*} + \lambda\boldI\right)^{-1}\mathbbE_+[\varphi(\boldx)]\|, \|\left(\boldSigma_{v^*} + \lambda\boldI\right)^{-1}\mathbbE_-[\varphi(\boldx)]\|\right)}}^{-1},$ we have 
    \begin{align}
        \label{eq.proof.bound.b}
        \textcolor{red}{a \le \frac{\|\tilde{u}\|^2}{2}}.
    \end{align}

    Now we inspect \textcolor{blue}{$b$}. 
    We can see $|{\color{blue}{b}}| \le \left|\tilde{s}\hat{\boldSigma}^+_{\bar{v}}\tilde{u} - \mathbbE_+\left[\tilde{s}\hat{\boldSigma}^+_{\bar{v}}\tilde{u} \right] \right|+ \left|\tilde{s}\hat{\boldSigma}^-_{\bar{v}}\tilde{u} - \mathbbE_-\left[\tilde{s}\hat{\boldSigma}^-_{\bar{v}}\tilde{u}\right]\right|$.
    Define a scalar random variable \[Z_f^{(i)} := \sin\langle v,\varphi(\boldx_i) \rangle \cdot \langle \tilde{s},  \varphi(\boldx_i) \otimes \varphi(\boldx_i) \tilde{u} \rangle.\] 
    By definition $\frac{1}{n_+}\sum_{i=1}^{n_+}Z^{(i)}_f = \tilde{s}^\top \hat{\boldSigma}^+_f  \tilde{u}.$
    Therefore \[\left |\frac{1}{n_+}\sum_{i=1}^{n_+} Z^{(i)}_{\bar{v}} - \mathbb{E} Z_{\bar{v}}\right| \le \sup_v \left|\frac{1}{n_+}\sum_{i=1}^{n_+} Z^{(i)}_v - \mathbb{E} Z_v\right|.\] Since $0\le Z^{(i)}_v \le  \|\tilde{s}\|\cdot\|\tilde{u}\|\cdot \|\varphi(\boldx)\|^2 \le \|\left(\boldSigma_{v^*}+\lambda\boldI\right)^{-1}\tilde{u}\| \|\tilde{u}\|\le \frac{\|\tilde{u}\|^2}{\lambda}$, using Uniform Law of Large Number  for bounded random variable (Theorem 4.10, \cite{wainwright_2019}), 
    \[
        \sup_v \left |\frac{1}{n_+}\sum_{i=1}^{n_+} Z^{(i)}_v - \mathbb{E} Z_v \right| \le 2\mathcal{R}_{n_+}(\mathcal{F}_Z) + \frac{\|\tilde{u}\|^2\cdot \|\varphi(\boldx)\|^2}{\lambda\sqrt{n_+}},
    \]
    with high probability, where $\mathcal{R}_{n_+}(\mathcal{F}_Z)$ is the Rademacher complexity of the function class of $Z_v$. 
    It remains to bound $\mathcal{R}_{n_+}(\mathcal{F}_Z)$. It can be seen that $Z_f = h[\langle v, \varphi(\boldx)\rangle]$ where $h$ is a Lipschitz  continuos function with Lipschitz constant $\frac{\|\tilde{u}\|^2}{\lambda}$. Hence, due to Ledoux–Talagrand contraction inequality (see, e.g., (5.61) in \cite{wainwright_2019}), $\mathcal{R}_{n_+}(\mathcal{F}_Z)$ is upperbounded by,
    \begin{align*}
        \mathcal{R}_{n_+}(\mathcal{F}_Z) \le \frac{2\|\tilde{u}\|^2}{\lambda} \cdot \mathcal{R}_{n_+}(\mathcal{H}^*) \le \frac{C_0 \cdot \|\tilde{u}\|^2}{\lambda\sqrt{n_+}},
    \end{align*}
    where $C_0$ is a universal constant. The last inequality is due to Corollary 14.5 in \cite{wainwright_2019}. 
    Therefore 
    \begin{align*}
        \left|\tilde{s}\hat{\boldSigma}^+_{\bar{v}}\tilde{u} - \mathbbE_+\left[\tilde{s}\hat{\boldSigma}^+_{\bar{v}}\tilde{u} \right] \right| \le   \frac{C_0 \cdot \|\tilde{u}\|^2}{\lambda\sqrt{n_+}} + \frac{\|\tilde{u}\|^2\cdot \|\varphi(\boldx)\|^2}{\lambda\sqrt{n_+}} 
    \end{align*}
    and similarly, 
    \begin{align*}
        \left|\tilde{s}\hat{\boldSigma}^-_{\bar{v}}\tilde{u} - \mathbbE_-\left[\tilde{s}\hat{\boldSigma}^-_{\bar{v}}\tilde{u} \right] \right| \le    \frac{C_0 \cdot \|\tilde{u}\|^2}{\lambda\sqrt{n_-}} + \frac{\|\tilde{u}\|^2\cdot \|\varphi(\boldx)\|^2}{\lambda\sqrt{n_-}}.
    \end{align*}
    Therefore, 
    \begin{align}
        \label{eq.proof.bound.c}
        \textcolor{blue}{|b| \le  \frac{C_0 \cdot \|\tilde{u}\|^2}{\lambda\sqrt{n_\mathrm{min}}} + \frac{\|\tilde{u}\|^2\cdot \|\varphi(\boldx)\|^2}{\lambda\sqrt{n_\mathrm{min}}}},
    \end{align}
    with high probability.
    Substituting 
    \textcolor{red}{\eqref{eq.proof.bound.b}}and \textcolor{blue}{\eqref{eq.proof.bound.c}} into \eqref{eq.proof.kkt2}, we get 
    \begin{align*}
        \langle \tilde{s}^\top, -\nabla_v \ell(v^*) - \lambda v^* \rangle + \textcolor{blue}{\frac{\max\left(C_0,\|\varphi(\boldx)\|^2\right)\cdot \|\tilde{u}\|^2} {\lambda \sqrt{n_\mathrm{min}}}} &\ge \|\tilde{u}\|^2 - \textcolor{red}{\frac{1}{2}\|\tilde{u}\|^2}. 
    \end{align*}
    Using triangle inequality and Hölder's inequality, we have 
    \begin{align}
        \label{eq.proof.kkt3}
        \textcolor{OliveGreen}{-\langle \tilde{s}, \nabla_v \ell(v^*)\rangle} + \|\tilde{u}\| \textcolor{magenta}{\|( \boldSigma_{v^*} + \lambda\boldI)^{-1} \lambda v^* \|}  + \frac{\max\left(C_0,1\right) \|\tilde{u}\|^2} {\lambda \sqrt{n_\mathrm{min}}}   &{\ge} \frac{1}{2} \|\tilde{u}\|^2.
    \end{align}
    Due to Assumption \ref{ass0}, $ \mathbbE[\nabla_v\ell(v^*)]=0$.
    Hence, 
    \begin{align*}
    \textcolor{OliveGreen}{-\langle \tilde{s}, \nabla_v \ell(v^*)\rangle} =& -\langle \tilde{s}, \mathbbE[\nabla_v \ell(v^*)] \rangle - \langle \tilde{s}, \nabla_v \ell(v^*) - \mathbbE[\nabla_v \ell(v^*)]\rangle  \\
    =& 0 - \langle \tilde{s}, \nabla_v \ell(v^*) - \mathbbE[\nabla_v \ell(v^*)]\rangle.
    \end{align*}
    We can see 
    \begin{align}
        \label{eq.proof.1}
        |\langle \tilde{s}, \nabla_v \ell(v^*) - \mathbbE[\nabla_v \ell(v^*)]\rangle| \le \frac{\| \tilde{u}\|}{\lambda} \|\nabla_v \ell(v^*) - \mathbbE[\nabla_v \ell(v^*)]\| \le  \textcolor{OliveGreen}{\frac{C_1\|\tilde{u}\|}{\lambda \sqrt{n_\mathrm{min}}}}
    \end{align}
    holds with high probability and $C_1$ is a universal constant (due to Lemma \ref{lem.probability.bound}).

    Moreover, since $v^* \in \mathcal{R}(\boldSigma^\beta_{v^*})$, there exists $g\in \mathcal{H}$, $v^* = \boldSigma_{v^*}^{\beta}g$. Notice $\boldSigma_{v^*}$ is a bounded, compact, self-adjoint linear operator (see Section \ref{sec:proof.operator.properties}). Therefore, Hilbert-Schmidt Theorem indicates, 
    $\boldSigma_{v^*} = \sum_i \alpha_i \psi_i \langle \psi_i, \cdot \rangle$, where $\psi_i, \alpha_i$ are eigenfunctions and eigenvalues of $\boldSigma_{v^*}$ respectively.
    Hence, 
    \begin{align}
    \label{eq.proof.3}
    \|( \boldSigma_{v^*} + \lambda\boldI)^{-1} v^* \lambda  \| 
    =
    \|( \boldSigma_{v^*} + \lambda\boldI)^{-1} \boldSigma_{v^*}^{\beta}g \lambda  \| 
    \le&   \left\|\sum_i \langle \boldpsi_i, g\rangle  \boldpsi_i \cdot \frac{ \alpha^\beta_i \lambda}{\alpha_i + \lambda} \right\| \notag \\
    \le& \lambda^\beta \left\|\sum_i \langle \boldpsi_i, g\rangle  \right\| 
    \le \textcolor{magenta}{\|{\boldSigma_{v^*}}^{-\beta} v^*\|\cdot \lambda^\beta}. 
    \end{align}
    Combine \eqref{eq.proof.kkt3}, \eqref{eq.proof.1} and \eqref{eq.proof.3} and cancel $\|\tilde{u}\|$, we can conclude that 
    \[      
        \textcolor{OliveGreen}{\frac{C_1}{\lambda \sqrt{n_\mathrm{min}}}} +  \textcolor{magenta}{\|{\boldSigma_{v^*}}^{-\beta} v^*\|\cdot \lambda^\beta} + \textcolor{blue}{\frac{\max\left(C_0,\|\varphi(\boldx)\|^2\right)} {\lambda \sqrt{n_\mathrm{min}}}}  \ge \frac{1}{2}\|\tilde{u}\|,
    \]
    with high probability.
    Set $\lambda = \frac{\max(C_1, C_0,1)}{n_\mathrm{min}^{1/4}}$, we have 
    \begin{align*}
        \frac{2}{n_\mathrm{min}^{1/4}} + \frac{\max(C_1, C_0,1)^\beta\|{\boldSigma_{v^*}^{-\beta}} v^*\|}{n_\mathrm{min}^{\beta/4}} &\ge \frac{1}{2}\|\tilde{u}\|,
    \end{align*}
    holds with high probability. Therefore, $\exists N_2$, when $n_\mathrm{min}>N_2$, $\|\tilde{u}\| = O_p(n_\mathrm{min}^{-\beta/4})$.
 
    Since $\|\tilde{u} \|= o_p(1)$, as long as $\delta \ge K\cdot n_\mathrm{min}^{-\beta/4}$ where $K>0$ is a constant, there exists a constant $N$ such that, when $n_\mathrm{min}>N$, $\tilde{v}$ is in the interior of $\mathcal{H}^*$ with high probability.
    When this happens, the constraint $v \in \mathcal{H}^*$ is no longer active. 
    This means $\tilde{v}$ is a stationary point of the objective function in \eqref{eq.obj.cons}. Moreover, $\tilde{v} \in \mathcal{H}^*$, so it is in the feasible region of \eqref{eq:obj_rkhs} thanks to Assumption \ref{ass.local.region}. 
    This further indicates that $\tilde{v}$ is also a solution to \eqref{eq:obj_rkhs}. 
    As \eqref{eq:obj_rkhs} is a strictly convex optimization problem, $\tilde{v}$ is also its only solution. Therefore $\tilde{v} = \hat{v}$ and $\|\hat{v} - v^*\| = \|\tilde{v} - v^*\| = O_p(n_\mathrm{min}^{-\beta/4})$.
\end{proof}

\begin{lem}
    \label{lem.probability.bound}
    Given any $v^*\in \mathcal{H}$ such that $\mathbbE[\nabla_v \ell(v^*)] = 0$, if $\| \varphi(\boldx)\|_\mathcal{H} \le B$ then
    \begin{align*}
        P(\|\nabla_v \ell(v^*)\|_\mathcal{H} > \delta) \le 4\exp\left(-\frac{n_\mathrm{min}\delta^2}{B^2}\right).
    \end{align*}
\end{lem}
\begin{proof}    
    Write down the definition of $\nabla_v \ell(v^*)$. Notice
    \begin{align*}
        \nabla_v \ell(v^*) &= \underbrace{-\frac{1}{n_+} \sum_{i=1}^{n_+}\cos\langle v, \varphi(\boldx_i) \rangle \varphi(\boldx_i)}_{a}  + \underbrace{\frac{1}{n_-} \sum_{i=1}^{n_-} \sin \langle v, \varphi(\boldx_i) \rangle \varphi(\boldx_i)}_{b}.
    \end{align*}
    By using Hilbert-space Hoeffding's inequality \cite{rosasco2010learning}, we know for all $\delta_a, \delta_b >0 $
    \begin{align*}
        P(\|a - \mathbbE[a]\|_\mathcal{H} > \delta_a) \le 2\exp\left(-\frac{Cn_+\delta_a^2}{B^2}\right) \text{ and } P(\|b - \mathbbE[b]\|_\mathcal{H} > \delta_b) \le 2\exp\left(-\frac{Cn_-\delta_b^2}{B^2}\right), 
    \end{align*}
    where $C$ is a constant. 
    Let $\delta = \delta_a = \delta_b$,
    \begin{align*}
        P(\|a+b\|_\mathcal{H} > 2\delta) =& P(\|a+b - (\mathbbE[a]+\mathbbE[b])\|_\mathcal{H} > 2\delta) \\
        \le&   P(\|a - \mathbbE[a]\|_\mathcal{H} + \|b - \mathbbE[b]\|_\mathcal{H} > \delta_a + \delta_b)\\
        \le&  P(\|a - \mathbbE[a]\|_\mathcal{H}  > \delta_a) + P(\|b - \mathbbE[b]\|_\mathcal{H} > \delta_b) \\
        \le&  4\exp\left(-\frac{Cn_\mathrm{min}\delta^2}{B^2}\right),
    \end{align*}
    where the first equality used the condition that $\mathbb{E}[\nabla_v \ell(v^*)] = \mathbbE[a] + \mathbbE[b] =0 $. 
    This completes the proof.  
\end{proof}

\section{Proof of Proposition \ref{prop.optimal.score}}
\label{sec:sec.proof.optimal.score}
\begin{proof}
    We start from the definition of $\mathbbE[\nabla_v \ell(v^*)]$:
    \begin{align*}
        -\mathbbE[\nabla_v \ell(v^*)] 
        =& \mathbbE_{+} \left[\frac{1}{n_+} \sum_{i=1}^n\cos\langle v^*, \varphi(\boldx_i) \rangle \varphi(\boldx_i) \right] - \mathbbE_{-} \left[\frac{1}{n_-} \sum_{i=1}^n \sin \langle v^*, \varphi(\boldx_i) \rangle \varphi(\boldx_i) \right]\\
        =& \mathbbE_{+} \left[\cos\langle v^*, \varphi(\boldx) \rangle \varphi(\boldx)\right] - \mathbbE_{-} \left[\frac{\sin \langle v^*, \varphi(\boldx) \rangle}{\cos\langle v^*, \varphi(\boldx) \rangle}\cos \langle v^*, \varphi(\boldx) \rangle \varphi(\boldx)\right]\\
        =& \mathbbE_{+} \left[\cos\langle v^*, \varphi(\boldx) \rangle \varphi(\boldx) \right] - \mathbbE_{-} \left[ \frac{p_+(\boldx)}{p_-(\boldx)} \cos\langle v^*, \varphi(\boldx)  \rangle \varphi(\boldx) \right]\\ 
        =& \mathbbE_{+} \left[\cos\langle v^*, \varphi(\boldx) \rangle \varphi(\boldx)  \right]  - \mathbbE_{+} [\cos\langle v^*, \varphi(\boldx) \rangle \varphi(\boldx) ]= 0,
    \end{align*}
    where the third equality is due to the fact that $\langle v^* , \varphi(\boldx)\rangle = \atan \frac{p_+(\boldx)}{p_-(\boldx)}$. 
    Since $p_+/p_- \in [0, \infty)$ , $\langle v, \varphi(\boldx) \rangle \in [0,\pi/2)$. 
    As $v^*$ is unique by assumption, Assumption \ref{ass0} holds.
\end{proof}

\section{Properties of Operator $\boldSigma_{v_0}$}
\label{sec:proof.operator.properties}
By construction, it is easy to verify that $\boldSigma_{v_0}$ is self-adjoint. 

First, we prove that the integral operator 
\[\boldSigma_{v_0}u = \mathbbE_+[\sin\langle v_0, \varphi(\boldx)\rangle \varphi(\boldx) \cdot u(\boldx)] + \mathbbE_-[\cos\langle v_0, \varphi(\boldx)\rangle \varphi(\boldx) \cdot u(\boldx)], \]
is a bounded operator. For all $u\in \mathrm{Ball}(0,1),$ where $\mathrm{Ball}(0,1)$ is the unit ball in $\|\cdot\|_\mathcal{H}$,
\begin{align*}
    \|\boldSigma_{v_0}u\|_\mathcal{H} &\le \|\mathbbE_+[\sin\langle v_0, \varphi(\boldx)\rangle \varphi(\boldx) \cdot u(\boldx)]\|_\mathcal{H} + \|\mathbbE_-[\cos\langle v_0, \varphi(\boldx)\rangle \varphi(\boldx) \cdot u(\boldx)]\|_\mathcal{H}\\
    &\le \mathbbE_+[\|\sin\langle v_0, \varphi(\boldx)\rangle \varphi(\boldx)\|_\mathcal{H} \cdot \|u(\boldx)\|_\mathcal{H}] + \mathbbE_-[\|\cos\langle v_0, \varphi(\boldx)\rangle \varphi(\boldx)\|_\mathcal{H} \cdot \|u(\boldx)\|_\mathcal{H}]\\
    &\le \mathbbE_+[\| \varphi(\boldx)\|_\mathcal{H} \cdot \|u(\boldx)\|_\mathcal{H}] + \mathbbE_-[\|\varphi(\boldx)\|_\mathcal{H} \cdot \|u(\boldx)\|_\mathcal{H}]\\
    &\le \mathbbE_+[\| \varphi(\boldx)\|_\mathcal{H}] + \mathbbE_-[\|\varphi(\boldx)\|_\mathcal{H}].
\end{align*}
Hence, $\boldSigma_{v_0}$ is a bounded operator as long as $\| \varphi(\boldx)\|_\mathcal{H}$ is bounded. 

Second, we show $\boldSigma_{v_0}$ is trace class hence compact. Let $\boldpsi_i, i \in \mathbb{N}$ be an orthonormal basis in $\mathcal{H}$, then 
\begin{align*}
    &\sum_i \langle \boldpsi_i, \boldSigma_{v_0}\boldpsi_i \rangle \\
    = & \mathbbE_+[\sin\langle v_0, \varphi(\boldx)\rangle \sum_{i\in \mathbb{N}}\langle \boldpsi_i, \varphi(\boldx) \otimes \varphi(\boldx) \boldpsi_i\rangle ] + \mathbbE_-[\cos\langle v_0, \varphi(\boldx)\rangle \sum_{i\in \mathbb{N}}\langle \boldpsi_i, \varphi(\boldx) \otimes \varphi(\boldx) \boldpsi_i\rangle ]\\
    = & \mathbbE_+[\sin\langle v_0, \varphi(\boldx)\rangle \sum_{i\in \mathbb{N}} \langle \boldpsi_i,\varphi(\boldx)\rangle^2 ] + \mathbbE_-[\cos\langle v_0, \varphi(\boldx)\rangle \sum_{i\in \mathbb{N}}\langle \boldpsi_i,\varphi(\boldx)\rangle^2 ]\\
    = & \mathbbE_+[\sin\langle v_0, \varphi(\boldx)\rangle \cdot \|\varphi(\boldx)\|^2_\mathcal{H} ] + \mathbbE_-[\cos\langle v_0, \varphi(\boldx)\rangle \cdot \|\varphi(\boldx)\|^2_\mathcal{H} ] < \infty
\end{align*}
holds as long as $\| \varphi(\boldx)\|_\mathcal{H}$ is bounded.
This shows $\boldSigma_{v_0}$ is trace-class and therefore, compact. 

\section{Proof of Proposition \ref{prop.auc.star}} 
\label{sec.proof.prop.auc}

\begin{proof}
Let us define for $\alpha \in [0,.5]$, 
    \[\tilde{F}^*_-(\cdot, \alpha) := 1 - [(1-\alpha) F^*_-(\cdot) + \alpha F^*_+(\cdot)], ~~~ \tilde{F}^*_+(\cdot, \alpha) := 1 - [\alpha F^*_- (\cdot) + (1-\alpha) F^*_+ (\cdot)]. \] 
    We can see that $\boldr(\tau, \alpha) := (\tilde{F}_-^*(\tau,\alpha), \tilde{F}_+^*(\tau, \alpha))$ is a parameterization for the space between $\mathrm{ROC}^*$ and the diagonal from $(0,0)$ to $(1,1)$. We can compute the surface area using the surface integral formula: 
    \[
        A_0 := \int_{\mathrm{dom}(\tau)} \int_{[0,.5]}  \|\partial_\tau\boldr(\tau, \alpha) \times \partial_\alpha\boldr(\tau, \alpha)\|  \ \mathrm{d}\alpha \mathrm{d} \tau,
    \]
    where $\partial_\tau\boldr(\tau, \alpha) = \begin{bmatrix} \partial_\tau \tilde{F}_-^*(\tau, \alpha)\\ \partial_\tau \tilde{F}_+^*(\tau, \alpha)\\
    0\end{bmatrix}$ and $\partial_\alpha\boldr(\tau, \alpha) =
    \begin{bmatrix} \partial_\alpha \tilde{F}_-^*(\tau, \alpha)\\ \partial_\alpha \tilde{F}_+^*(\tau, \alpha)\\
    0\end{bmatrix}$. 
    It can be seen that 
    $\partial_\alpha\boldr(\tau, \alpha) =
    \begin{bmatrix} F^*_-(\tau) - F^*_+(\tau)\\F^*_+(\tau) - F^*_-(\tau) \\0\end{bmatrix}$ for all $\alpha$. Rewrite $A_0$:
    \begin{align}
    \label{auc.cross.product}
            A_0 &= \int_{\mathrm{dom}(\tau)} \int_{[0,.5]}  \left| \left[ F^*_-(\tau) - F^*_+(\tau) \right] \partial_\tau \tilde{F}_+^*(\tau, \alpha) - \left[F^*_+(\tau) - F^*_-(\tau) \right] \partial_\tau\tilde{F}_-^*(\tau, \alpha)  \right|  \ \mathrm{d}\alpha \mathrm{d} \tau, \notag \\
            &= \int_{\mathrm{dom}(\tau)} \int_{[0,.5]}  \left| \left[F^*_-(\tau) - F^*_+(\tau) \right] (\partial_\tau F^*_+(\tau) + \partial_\tau F^*_-(\tau))  \right|  \ \mathrm{d}\alpha \mathrm{d} \tau, \notag \\
            &= \int_{\mathrm{dom}(\tau)} \int_{[0,.5]}  \| \bolda(\tau) \times \boldb(\tau)  \|  \ \mathrm{d}\alpha \mathrm{d} \tau,
    \end{align}
    where $\bolda(\tau) = \begin{bmatrix} F^*_-(\tau) - F^*_+(\tau)\\ F^*_+(\tau) - F^*_-(\tau) \\ 0 \end{bmatrix}$ and $\boldb(\tau) = \begin{bmatrix} \partial_\tau F^*_-(\tau) \\ \partial_\tau F^*_+(\tau)  \\ 0\end{bmatrix}$. Both $\bolda$ and $\boldb$ are free from $\alpha$. 
    Rewriting the cross product in \eqref{auc.cross.product} in a different form, we obtain 
        \begin{align}
            A_0 &= \sqrt{2}\int_{\mathrm{dom}(\tau)} \int_{[0,.5]} \sin(\theta(\tau)) \left|F^*_-(\tau) - F^*_+(\tau)\right| \sqrt{\partial_\tau F^*_+(\tau)^2 + \partial_\tau F^*_-(\tau)^2}    \ \mathrm{d}\alpha \mathrm{d}\tau, \notag \\
             &= \frac{\sqrt{2}}{2} \int_{\mathrm{dom}(\tau)} \sin(\theta(\tau))  \left|F^*_-(\tau) - F^*_+(\tau)\right| \sqrt{\partial_\tau F^*_+(\tau)^2 + \partial_\tau F^*_-(\tau)^2}  \ \mathrm{d}\tau, \notag \\
             &= \frac{\sqrt{2}}{2} \int_{\mathrm{dom}(\tau)} \sin(\theta(\tau)) f^*_ -(\tau) \left|F^*_-(\tau) - F^*_+(\tau)\right| \sqrt{\left(\frac{f^*_+(\tau)}{ f^*_-(\tau)}\right)^2 + 1}    \ \mathrm{d}\tau, 
    \end{align}
    where $\theta(\tau)$ is the angle between $\bolda(\tau)$ and $\boldb(\tau)$. $\boldb(\tau)$ is the tangent vector of the $\rockstar$. 
    Knowing the slope of $\rockstar$ is the likelihood ratio (see Section \ref{sec.fench}) and $\bolda(\tau)$ points at the 45 degree downward regardless of $\tau$,  we can see 
      $\theta(\tau) = \left[\atan \frac{p(\boldx)}{q(\boldx)}\right] + \frac{\pi}{4}$. Using the fact that $\frac{f^*_+(t(\boldx))}{ f^*_-(t(\boldx))} = \frac{p_+(\boldx)}{ p_-(\boldx)}$ and the law of unconscious statistician: 
    \begin{align*}
        A_0 = \frac{\sqrt{2}}{2}\mathbbE_{p_-}\left\{ \sin\left[\left(\atan \frac{p_+(\boldx)}{p_-(\boldx)}\right) + \frac{\pi}{4}\right]   \left|F^*_-(t^*(\boldx)) - F^*_+(t^*(\boldx))\right| \sqrt{\left(\frac{p_+(\boldx)}{ p_-(\boldx)}\right)^2 + 1} \right\}
    \end{align*}
    Replacing $\sqrt{\frac{p_+(\boldx)}{ p_-(\boldx)}^2 + 1}$ with its Fenchel dual as introduced in Section \ref{sec.fench} and pulling the $\sup$ out of the expectation yields the desired result. 
    
    Differentiating the objective \eqref{eq.weighted.arclength} with respect to $v$ and setting the derivative to zero, we can see that superemum is attained at $v^* = \atan\frac{p_+}{p_-}$.

\end{proof}

\section{Wall Clock Comparison}
\begin{figure}
  \centering
  \includegraphics[width=.5\linewidth]{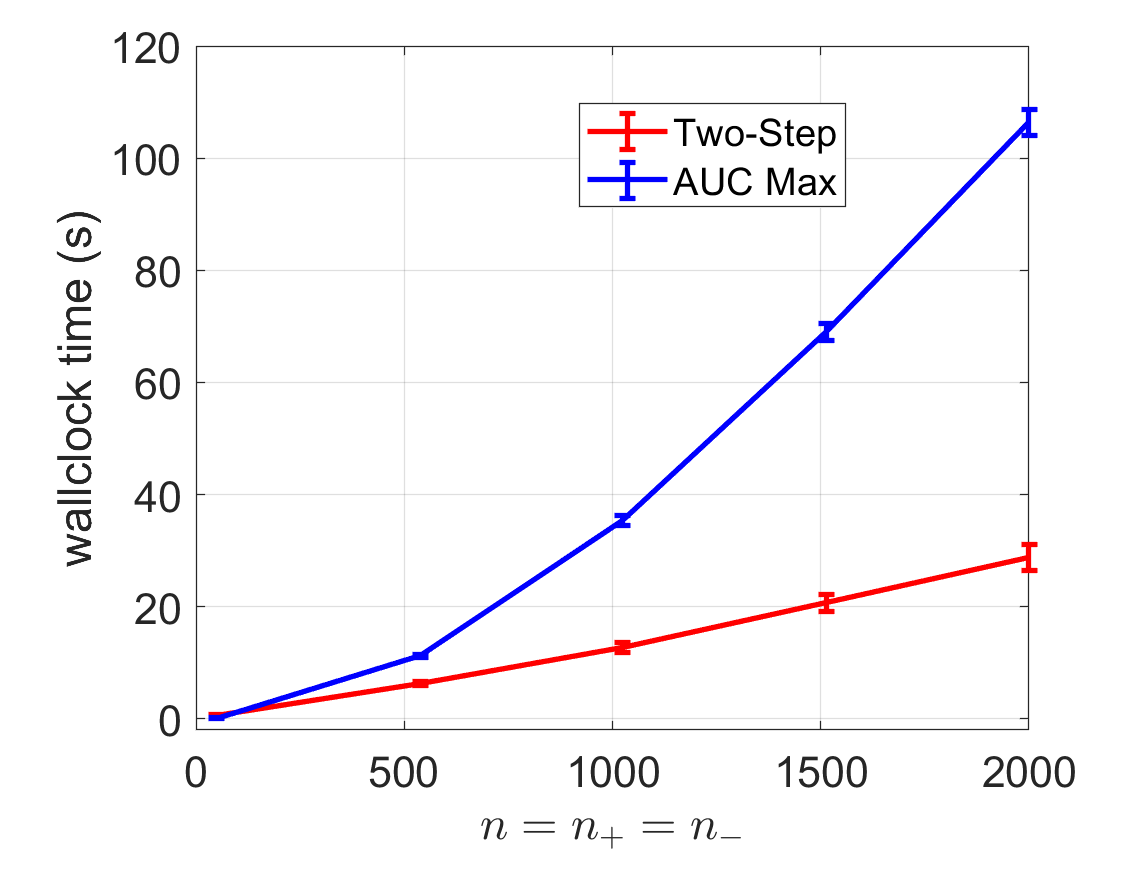}
  \captionof{figure}{The time comparison over different sample sizes $n=n_+=n_-$.}
  \label{fig:wallclock}
\end{figure}

\red{In this experiment, we evaluate the computation time of the our two-step algorithm and the naive implementation of the offline AUC maximization \eqref{eq.auc.max.empirical} by plotting the wall clock time in Figure \ref{fig:wallclock}. 
Both the AUC maximization and two-step procedure are implemented using MATLAB's optimization toolbox. 
See Section \ref{sec.exp.supp} for details. The two-step procedure's computation time grows at a much slower rate than the offline AUC maximization via a pairwise loss function. 
Note that as we explained in Section \ref{sec.comp.complexity}, if the surrogate loss is decomposable, the objective can be computed with a computational complexity $O(n\log(n))$ \citep{joachims2005support}. If it the loss is squared loss, the offline algorithm can be performed with a $O(n)$ computational complexity \citep{yingSOAUC}. 
} 

In this experiment, both methods are written in fully vectorized code. The first order derivatives are provided to the $\mathrm{fmincon}$ and $\mathrm{fminunc}$ to accelerate the computation. Code can be found in the supplementary material.

\section{Experiment Setup}
\label{sec.exp.supp}

In Section \ref{sec.exp}, we reduce the dimension of CIFAR-10 dataset to 50. We first train a residual neural network \cite{he2016deep} using logistic regression on all 10 classes. This 103-layer network structure was borrowed from a MATLAB tutorial (\url{https://www.mathworks.com/help/deeplearning/ug/train-residual-network-for-image-classification.html}). MATLAB provides a pretrained version of this network. To obtain bounded features, we append a fully connected linear layer (output dimension 50) and a bounded activation layer (clipped-relu) to the last average pooling layer in the network. We freeze the earlier layers and only train the last two layers for 5 epochs. 

The dataset and the code that reproduces Figure \ref{fig:cifar10} can be found in the supplementary materials. We invite reviewers to reproduce our results.  

\section{Estimating $\log\left[\frac{p_+(\boldx)}{p_-(\boldx)}\right]$} 
\label{sec.est.logratio}
We can also leverage that $v^*$ is the arctangent of the likelihood ratio and introduce mild assumptions on $p_+$ and $p_-$. 
When $p_+(\boldx)$ and $p_-(\boldx)$ are both members of the exponential family and share the same sufficient statistic $\boldh(\boldx) \in \mathbb{R}^m$, then $\exists \boldv^* \in \mathbbR^m$ such that $\log\left[\frac{p_+(\boldx)}{p_-(\boldx)}\right] = \langle \boldv^*, \boldh(\boldx)\rangle + C$, where $C$ is a constant. 
If we choose to parameterize the log likelihood ratio using a linear model, 
$\langle \boldv, \boldh(\boldx)\rangle + v_0$, 
then \eqref{eq.sample.obj} becomes
\begin{align}
    \label{eq.sample.obj.logratio}
            (\hat{\boldv}, \hat{v}_0) :=\argmax_{
            \substack{\boldv\in \mathbbR^m,  \\ v_0\in \mathbb{R}~~}} \frac{1}{n_+} \sum_{i=1}^{n_+} \sin[\atannospace\exp(\langle \boldv, \boldh(\boldx_i)\rangle + v_0)] 
            + \frac{1}{n_-} \sum_{i=1}^{n_-} \cos[\atannospace\exp(\langle \boldv, \boldh(\boldx_i)\rangle + v_0)].
\end{align}
Note we do not have to restrict the optimization to a bounded function family as $\log\left[\frac{p_+(\boldx)}{p_-(\boldx)}\right] \in \mathbb{R}$, and 
$\langle \hat{\boldv}, \boldh(\boldx)\rangle + \hat{v_0}$ is an estimate of the likelihood ratio. 

In this paper, we focus on \eqref{eq:obj_rkhs} since the objective in \eqref{eq.sample.obj.logratio} is non-convex with respect to $\boldv$ thus presents extra challenges in the theoretical analysis, although \eqref{eq.sample.obj.logratio} is easier to implement in practice due to its unconstrained nature.

\section{Numerical Simulation of $\atan\frac{p_+}{p_-}$ Estimation}
\red{
We draw 100 samples from $X_+ \sim \mathcal{N}(1, 1)$ and $X_- \sim \mathcal{N}(-1, 1)$ and solve \eqref{eq:obj_rkhs} to estimate the arctangent density ratio. The estimated arctangent density ratio with standard deviation (over 72 runs) are plotted in Figure \ref{fig.lengthapp}. 
We use Gaussian kernel and hyperparameters (kernel bandwidth and regularization parameter) are tuned using cross validation. 
}

\red{
We observe that the estimated arctangent ratio using the  proposed method is very close to the ground truth and has a small standard deviation. 
}

\begin{figure*}[t]
    \centering

    \includegraphics[width=.6\textwidth]{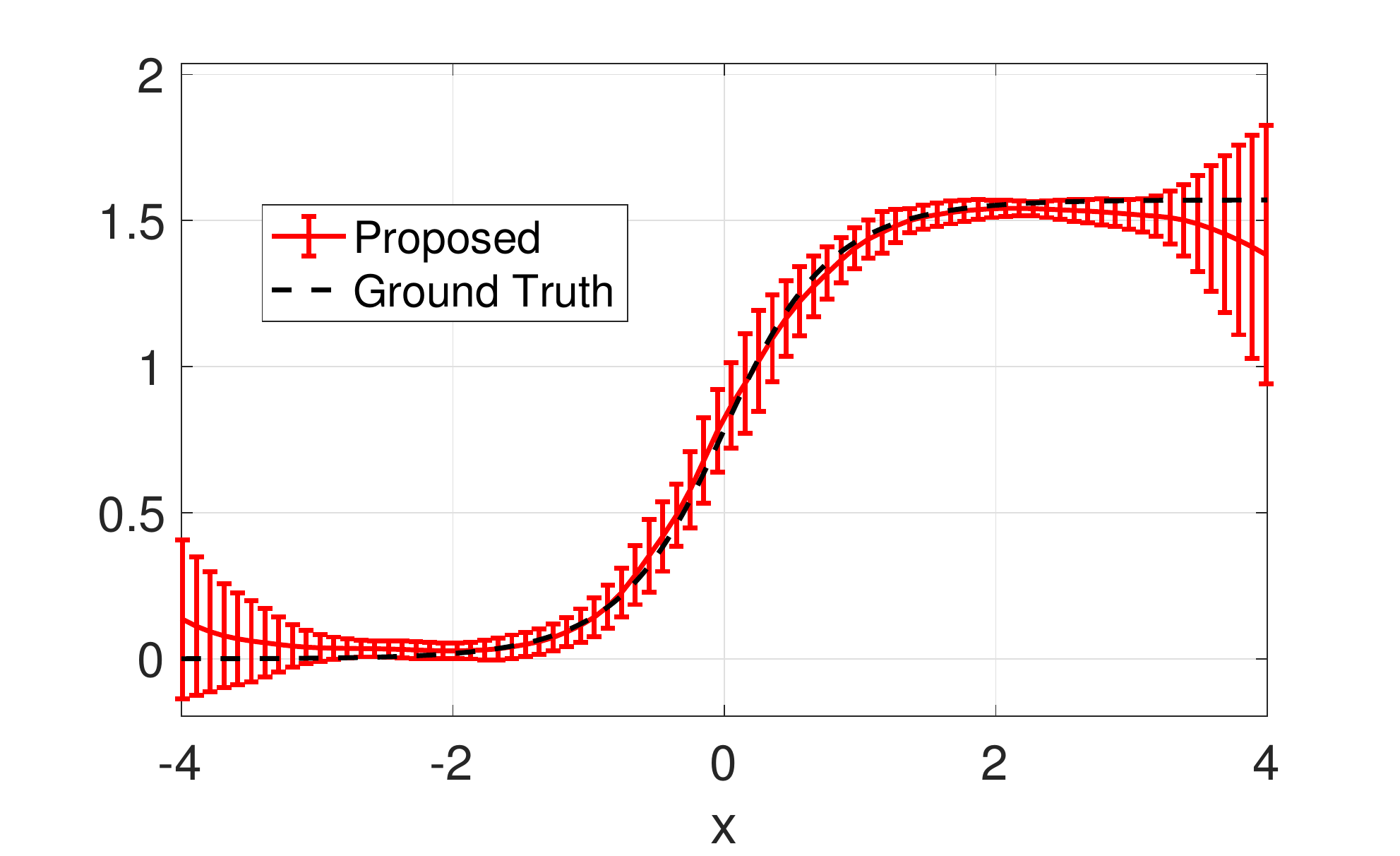}
\caption{Estimation of the arctangent density ratio function.}
\label{fig.lengthapp}
\end{figure*}

\section{Comparison with Convergence Result in \citet{Nguyen2008}}
\label{sec.comp.conv.results}
\red{
The convergence of (log) density ratio estimation have been developed for two KL divergence based estimators \cite{Nguyen2008}. However, these convergence theories are not general theories for arbitrary $f$-divergences. Thus their proofs cannot be easily applied to our ROC divergence.
}

\red{
Moreover, Theorem \ref{thm:base} is not a minor modification of convergence theories in \cite{Nguyen2008}. Specifically, \citet{Nguyen2008} prove the likelihood ratio converges in Hellinger distance, while we prove the arctangent likelihood ratio converges in Hilbert space norm. The proofs rely on completely different machinery and assumptions. These technical results depend on the variational objective functions the estimators maximize and are not interchangeable.
}
\end{document}